\documentclass[12pt]{amsart}

\usepackage[utf8]{inputenc}
\usepackage[english]{babel}

\usepackage{csquotes}
\usepackage{enumitem}
\usepackage{amsmath, amsthm, amsfonts, amssymb}
\usepackage{mathtools}
\usepackage[colorlinks=true, linkcolor=red, citecolor=green, backref=page]{hyperref}
\usepackage{color}
\usepackage{graphicx}
\usepackage[capitalise, noabbrev]{cleveref}
\usepackage{geometry}
\usepackage{marginnote}
\usepackage{tikz}
\usepackage{tikz-cd}
\usetikzlibrary{decorations.markings}
\usepackage{centernot}
\usepackage{chngcntr}
\usepackage{mathrsfs}
\usepackage{placeins}
\usepackage{stmaryrd}

\usepackage{algorithmic}

\newcommand{\overbar}[1]{\mkern 1.5mu\overline{\mkern-3mu#1\mkern-0.5mu}\mkern 0.5mu}
\newcommand{\dd}{\mathrm{d}}

\newcommand\restr[2]{{
  \left.\kern-\nulldelimiterspace 
  #1 
  \vphantom{\big|} 
  \right|_{#2} 
  }}

\DeclareMathOperator{\lcm}{lcm}

\newcommand{\jump}{\vskip 2mm}

\numberwithin{equation}{section}

\theoremstyle{plain}
\newtheorem{theorem}{Theorem}
\newtheorem{proposition}{Proposition}
\newtheorem{lemma}{Lemma}
\newtheorem{corollary}{Corollary}
\newtheorem*{conjecture}{Conjecture}

\theoremstyle{definition}

\newtheorem{remark}{Remark}
\newtheorem{example}{Example}
\AtBeginEnvironment{example}{\pushQED{\qed}}
\AtEndEnvironment{example}{\popQED\endexample}

\numberwithin{theorem}{section}
\numberwithin{proposition}{section}
\numberwithin{lemma}{section}
\numberwithin{corollary}{section}
\numberwithin{remark}{section}

\newcounter{introduction}

\newtheorem{intro-proposition}[introduction]{Proposition}
\newtheorem{intro-corollary}[introduction]{Corollary}
\newtheorem{intro-theorem}[introduction]{Theorem}
\newtheorem{intro-question}[introduction]{Question}

\title[Topological roots of the Bernstein-Sato polynomial of plane curves]{Topological roots of the Bernstein-Sato polynomial of plane curves}

\author[G. Blanco]{Guillem Blanco}

\thanks{}

\address{Departament de Matem\`atiques\\ Universitat Polit\`ecnica de Catalunya,
Diagonal 647 08028 Barcelona, Spain.}
\email{guillem.blanco@upc.edu}

\begin{document}

\begin{abstract}
We study a set of topological roots of the local Bernstein-Sato polynomial of arbitrary plane curve singularities. These roots are characterized in terms of certain divisorial valuations and the numerical data of the minimal log resolution. In particular, this set of roots strictly contains both the opposites of the jumping numbers in \((0, 1)\) and the poles of the motivic zeta function counted with multiplicity. As a consequence, we prove the multiplicity part of the Strong Monodromy Conjecture for \(n = 2\).
\end{abstract}

\maketitle

\section{Introduction} \label{intro}

Let \( f : (\mathbb{C}^2, x) \longrightarrow (\mathbb{C}, 0) \) be a germ of a holomorphic function defining a germ of a plane curve \( C \), not necessarily reduced or irreducible.

\jump

Any plane curve has a unique minimal log resolution that we denote by \( \pi : (Y, E) \longrightarrow (\mathbb{C}^2, x) \), where \( E \) is the exceptional locus. That is, \( \pi \) is a proper birational morphism satisfying that \( Y \) is a smooth complex manifold, \( \pi \) is an isomorphism outside \( \{x\} \) and \( \pi^{-1} C\) has simple normal crossings. Denote
\[ F_\pi = \textnormal{Div}(\pi^* f) = \sum_{i \in T} N_i E_i, \qquad K_\pi = \textnormal{Div}(\det \textnormal{Jac}\, \pi) = \sum_{i \in T} (k_i - 1) E_i, \]
the total transform divisor and relative canonical divisor, respectively. The numbers \( (N_i, k_i)_{i \in T} \) are usually called the numerical data of the log resolution \( \pi \). The set \( T \) is naturally partitioned as \( T_e \cup T_s \), where \( T_e \) runs over the exceptional components and \( T_s \) runs over the components of the strict transform of \( C \).

\jump

It is a classical result \cite{Bra28} that the embedded topological type of \( C \) is determined by the combinatorics of its minimal log resolution. Conversely, two plane curves having the same combinatorics in the minimal resolution, also known as \emph{equisingular}, have the same topological type, see \cite{Zar32}.

\jump

In this work, we are interested in the behavior of the roots of the local Bernstein-Sato polynomial within a topological class of plane curves. Recall that, in the local case, the Bernstein-Sato polynomial is the monic polynomial of the smallest degree in \( \mathbb{C}[s] \) fulfilling the function equation
\[ P(s) \cdot f^{s+1} = b_{f, x}(s) f^s, \]
where \( P(s) \) is a differential operator with coefficients in the ring \( \mathscr{O}_{\mathbb{C}^2, x} \otimes \mathbb{C}[s] \). The existence of the Bernstein-Sato polynomial is due to Bernstein \cite{Ber72} in the algebraic case and Bj\"ork \cite{Bjo74} in the analytic case. Since \(s = -1\) is always a root of \(b_{f, x}(s)\), one defines the reduced local Bernstein-Sato polynomial \(\widetilde{b}_{f,x}(s) = b_{f,x}(s) / (s + 1)\).

\jump

It is well-known that \( b_{f, x}(s) \) is not a topological invariant, that is, some of its roots can change within a topological class of plane curves, see \cite{Yano78} for some examples. Therefore, a number \( \sigma \in \mathbb{Q}_{<0} \) is called a \emph{topological root} of the (local) Bernstein-Sato polynomial for a topological class \( \mathcal{T} \) of plane curves if for every germ of a curve \( (f^{-1}(0), x) \) in \( \mathcal{T} \), \( \sigma \) is a root of the local Bernstein-Sato polynomial \( b_{f, x}(s) \) of \( f \).

\jump

Given an exceptional divisor \( E_i \) from the minimal log resolution of \( C \) we will denote by \( r_i \) the number of divisors, different from \( E_i \), in the support of \( F_\pi \) that intersect \( E_i \). An exceptional divisor is called \emph{rupture} if \( r_i \geq 3 \). \( E_i \) is called a \emph{satellite} divisor if the center of the blow-up producing \( E_i \) is in the intersection of two exceptional divisors. The combinatorics of the divisors appearing in a log resolution are encoded in a tree-shaped graph \(\Delta_\pi\) called the \emph{dual graph} of \(\pi\), see \cref{dual-graph}.

\jump

Let \( v_i: \mathscr{O}_{\mathbb{C}^2, x} \longrightarrow \mathbb{Z}_{\geq 0} \) be the divisorial valuation associated with \( E_i \) and denote by \( C_i \) a germ of a branch such that its strict transform intersects \( E_i \) transversally at a general point. For any \( g \in \mathscr{O}_{\mathbb{C}^2, x} \), let \( \omega = g \dd x \) and define
\[ \sigma_{i}(\omega) = - \frac{k_i + v_i(g)}{N_i}. \] 
After results of Kashiwara \cite{Kas76} and Lichtin \cite{Lich89}, the rational numbers \( \sigma_i(\omega) \) are a subset of the candidate roots for the local Bernstein-Sato polynomial of \( f \). The set of roots of \(b_{f,x}(s)\) that are of the form \( \sigma_{i}(\omega) \) for some differential form \( \omega \in \Omega^2_{\mathbb{C}^2, x} \) and some exceptional divisor \(E_i, i \in T_e,\) will be called roots of \emph{geometric origin}. 

\jump

Let \( E_{i_j}, i_j \in T, j = 1, \dots, m_i\), be the divisors in the support of \( F_\pi \) that are adjacent to \(E_i\). We define the \emph{residue numbers} of the differential form \(\omega\) at \(E_i\) as
\begin{equation} \label{eq-definition-eps}
    \epsilon_{i_j}(\omega) = N_{i_j} \sigma_i(\omega) + k_{i_j} + v_{i_j}(g), \qquad j = 1, \dots, m_i.
\end{equation}
To avoid distinguishing cases we allow some divisors \(E_{i_j}\) to be zero, in which case the residue numbers are equal to one, see \cref{sec-neigh-valuations} for the exact conventions. The residue numbers appear when studying the residues of archimedean and non-archimedean zeta functions, as well as asymptotic expansions of period integrals in the Milnor fiber; see \cite{Meu83, Lich85, Loe88, Loe90, Vey91, Bla19, Bla21} and \cref{expansions}.

\subsection{Main results}

The main result of this work is to characterize the topological roots of the local Bernstein-Sato polynomial that are of geometric origin in terms of  divisorial valuations associated to rupture divisors and the residue numbers at that divisor.

\jump

The bounds for the residue numbers obtained in the following theorem will be key to characterize the topological roots of geometric origin in \cref{thm-main-2}. The exceptional divisors \(E_{i_1}, E_{i_2}\) in the statement of the theorem denote the divisors intersecting \(E_i\) that are in the paths of \(\Delta_\pi\) connecting \(E_i\) with divisors preceding \(E_i\) in the log resolution, see \cref{sec-neigh-valuations} for the adopted conventions. 

\begin{intro-theorem} \label[theorem]{thm-main-1}
    Let \( E_i \) be an exceptional divisor and let \(\omega \in \Omega^2_{\mathbb{C}^2, x}\) be a top differential form. Then, there always exists \( \omega' \in \Omega_{\mathbb{C}^2, x}^2 \) with \( \sigma_i(\omega') = \sigma_i(\omega) \) and such that
    \[ -1 \leq \epsilon_{i_j}(\omega') \leq 1, \qquad \textnormal{for all} \qquad E_{i_j} \cap E_i \neq \emptyset. \]
    Moreover, \( \epsilon_{i_j}(\omega') = -1 \) only occurs if \( r_i = 1, r_i = 2 \), or \( r_i = 3 \) with \( E_i \) satellite and
    \begin{equation} 
        N_i = \lambda N_i(C_i), \quad N_{i_j} = \lambda N_{i_j}(C_i), \qquad \textnormal{for some}\qquad \lambda \in \mathbb{Z}_{>0}. \tag{\dag}
    \end{equation}
    Similarly, \(\epsilon_{i_j}(\omega') = 1 \) only occurs for \(E_{i_1}\), resp. \( E_{i_2}\), if there are no arrowheads in the connected component of \(\Delta_\pi \setminus \{ E_i \} \) containing \(E_{i_1}\), resp. \(E_{i_2}\).
\end{intro-theorem}

This result can be seen as a generalization of \cite[Prop. II.3.1]{Loe88} which can be recovered by letting \( \omega \) be the standard volume form. The exception for the case \(\epsilon_{i_j}(\omega') = -1\) and \( r_i = 3 \) in the statement of \cref{thm-main-1} can be interpreted as the fact that \( C \) does not look like a power of an irreducible plane curve locally around the exceptional divisor \( E_i \). This can be related to the fact that the monodromy of irreducible plane curves is semisimple \cite{Tra72}. For simplicity, we will refer to this exceptional case by \((\dag)\).

\begin{intro-theorem} \label[theorem]{thm-main-2}
    Let \(E_i\) be a rupture divisor and let \(\omega = g \dd x \in \Omega^2_{\mathbb{C}^2, x}\) be a top differential form such that \(0 \leq v_i(g) < N_i\) and satisfying \cref{thm-main-1}. Then, 
    \begin{enumerate}[label=\roman*)]
        \item \( \sigma_i(\omega) \) is a topological root of \( b_{f, x}(s) \) with multiplicity \( 2 \), if there is at least one \( \epsilon_{i_j}(\omega) = 0 \), for some \(j \in \{ 1, \dots, m_i\}\),
              \vskip 1 mm
        \item \( \sigma_i(\omega) \) is a topological root of \( b_{f, x}(s) \) with multiplicity \( 1 \), if \(\epsilon_{i_j}(\omega) \neq 0\), for \(j = 1, \dots, m_i\), and there are at least three \( \epsilon_{i_j}(\omega) \in \mathbb{Q} \setminus \mathbb{Z} \), for some \(j \in \{ 1, \dots, m_i\}\),
    \end{enumerate}
    and the conclusion is independent of the differential form satisfying \cref{thm-main-1}. Additionally, if \( E_i \) is a non-exceptional divisor with multiplicity \( N_i \), then \( -\frac{1}{N_i},\dots, -\frac{N_i - 1}{N_i}, -1 \) are topological roots of geometric origin of \( b_{f, x}(s) \).
\end{intro-theorem}

 Roots contributed by non-exceptional divisors will be topological by setting $g$ equal to suitable powers of the irreducible factors of $f$. See \cref{rmk-multiplicity} for a discussion about the possible multiplicities of the roots coming from non-exceptional divisors. Similar sets of topological roots of the local Bernstein-Sato polynomial for irreducible plane curves with one or two Puiseux pairs are studied in \cite{CN87} and \cite{ACLM17b}. For a closed expression for the topological roots from \cref{thm-main-2} when the curve is irreducible see \cref{cor-thm-2}.

\jump

The easiest nontrivial example of a topological root of the local Bernstein-Sato polynomial is the opposite of the log canonical threshold of \( f \) at \(x \in \mathbb{C}^2\). More generally, it follows from the birational description of the multiplier ideals, see \cite{Laz04-2} or \cref{sec-multipliers}, that the jumping numbers are always topological invariants of plane curve singularities. The fact that the opposites of the jumping numbers in \((0, 1)\) are always roots of \(b_{f,x}(s)\) follows from the results in \cite{ELSV04, BS05}.

\jump

The Monodromy Conjecture was first formulated by Igusa \cite{Igu88} in the \(p\)-adic setting, and later stated by Denef and Loeser \cite{DL92} for complex varieties. The strong version of the Monodromy Conjecture asserts that the poles of the local motivic zeta function \(Z_{\textnormal{mot}, x}(f; s)\) are roots of \(b_{f, x}(s)\), see \cref{sec-motivic}. The description of \(Z_{\textnormal{mot}, x}(f; s)\) in terms of a log resolution implies that, for plane curves, its poles are topological invariants. After the work of Loeser \cite{Loe88} on the Strong Monodromy Conjecture in dimension 2, the poles of \(Z_{\textnormal{mot}, x}(f;s)\) are always roots of \(b_{f,x}(s)\).

\jump

However, for a general topological class, the intersection of these two sets of topological roots is just the opposite of the log canonical threshold at \(x \in \mathbb{C}^2\). \cref{thm-main-2} gives a unified description of both sets.

\begin{intro-proposition} \label[proposition]{main-prop}
    The set of topological roots of the local Bernstein-Sato polynomial from \cref{thm-main-2} contains both the opposites of the jumping numbers in \( (0, 1) \) and the poles of the local motivic zeta function \( Z_{\textnormal{mot}, x}(f; s) \) counted with multiplicities.
\end{intro-proposition}

The Strong Monodromy Conjecture also predicts that the order of a pole of \(Z_{\textnormal{mot}, x}(f; s)\) is less than or equal to its multiplicity as a root of \(b_{f, x}(s)\). In dimension 2, this second part of the Strong Monodromy Conjecture is not completely settled in \cite{Loe88} if \(f\) is not reduced, see the discussions in \cite{MHTV09, MHTV10} regarding the difference between the complexes of vanishing and nearby cycles for non-isolated singularities. In \cite{Loe88}, when \(f\) is reduced, the singularity at \(x \in \mathbb{C}^2\) is isolated and one can obtain enough information about the roots of \(b_{f,x}(s)\) by directly using the minimal polynomial of the monodromy at \(x\) to construct suitable vanishing cycles in the homology \(H_1(X_t, \mathbb{C})\) of the Milnor fiber. 

\jump

The key point of our approach is that one has to allow relative cycles in \(H_1^{lf}(X_t \setminus \partial X_t, \mathbb{C})\) to detect all multiplicities greater than one. Moreover, we can explicitly construct both these relative cycles and the classical vanishing cycles giving rise to roots of \(b_{f, x}(s)\) with multiplicity greater than one. The characterization of the multiplicity is done by constructing the asymptotic expansions of period integrals along these cycles and without resorting to the local monodromy of \(f\), see \cref{prop-log-cycles}.

\begin{intro-corollary}
    For any plane curve singularity, the order of a pole of \( Z_{\textnormal{mot}, x}(f; s) \) is at most its multiplicity as a root of \(b_{f, x}(s)\).
\end{intro-corollary}

Finally, it is possible to find examples where the set of topological roots from \cref{thm-main-2} are all possible topological roots of the local Bernstein-Sato polynomial, see \cref{examples}. It is then reasonable to ask whether the set from \cref{thm-main-2} contains all possible topological roots of \( b_{f, x}(s) \) for a given topological class, see \cref{the-question}.

\jump

\textbf{Acknowledgments.} 
The author would like to thank Maria Alberich-Carramiñana and Josep Àlvarez Montaner for the fruitful discussions that originated this work. He would also like to thank Wim Veys for several conversations about the results of this work, especially for providing the proof of \cref{lemma-valuations-1} and pointing out the gap in the literature regarding the multiplicity part of the Strong Monodromy Conjecture in dimension \(2\).

\jump

The author was supported by a Postdoctoral Fellowship of the Research Foundation -- Flanders (FWO).

\section{Numerical data of plane branches} \label{semigroup}

For this section let \( f : (\mathbb{C}^2, x) \longrightarrow (\mathbb{C}, 0) \) be a germ of a holomorphic function defining an irreducible plane curve \( C \), also called a plane branch.

\subsection{Puiseux series} \label{sec-puiseux-pairs}

To any plane branch, we associate the set of characteristic exponents \( \{\beta_1/n, \beta_2/n, \dots, \beta_g/n\} \) that are defined from a Puiseux series of \( f \). Let \( s(x) \) be any Puiseux series of \( f \), then 
\begin{equation} \label{puiseux-series}
    s(x) = \sum_{\substack{j \in (e_0)\\ 1 \leq j < \beta_1}} a_j x^{j/n} + \sum_{\substack{j \in (e_1) \\ \beta_1 \leq j < \beta_2}} a_j x^{j/n} + \cdots + \sum_{\substack{j \in (e_{g-1})\\ \beta_{g-1} \leq j < \beta_g}} a_j x^{j/n} + \sum_{j \geq \beta_g} a_j x^{j/n},
\end{equation}
where all coefficients \( a_{\beta_i} \) are different from zero and \( e_i = \gcd(e_{i-1}, \beta_i), e_0 = n \). After an analytic change of coordinates we can always assume \( n < \beta_1 < \cdots < \beta_g \), and \( n \) is then the multiplicity of \( f \) at \( x \). Set \( n_i = e_{i-1}/e_i \), for \( i = 1, \dots, g \), and \( \beta_0 = 0, n_0 = 0 \) by convention. For \(i = 1, \dots, g\), these integers are strictly larger than \( 1 \), and we have that \( e_{i-1} = n_i n_{i+1} \cdots n_g \). In particular, \( n = n_1 \cdots n_g \). The fractions \( m_i / n_1 \cdots n_i \), with \( m_i = \beta_i / e_i \), are called the reduced characteristic exponents. Define
\begin{equation} \label{eq-char2puiseux}
     q_i = m_i - n_i m_{i-1}, \quad \textnormal{for}\ i = 1, \dots, g.
\end{equation}
Then, the pairs \( (q_i, n_i), i = 1, \dots, g \), are called the Puiseux pairs of \( C \), and they satisfy \( \gcd(q_i, n_i) = 1\). A Puiseux series \( s(x) \) of \( f \) can then be written as
\begin{equation*}
    \begin{split}
        s(x) &= \sum_{1 \leq j \leq [q_1/n_1]} a_{0,j} x^i \\
        &+ \sum_{0 \leq j \leq [q_2/n_2]} a_{1,j} x^{(q_1 + j)/n_1} \\
        &+ \sum_{0 \leq j \leq [q_3/n_3]} a_{2,j} x^{q_1/n_1 + (q_2 + j)/n_1n_2} \\
        &\ \, \vdots \\
        &+ \sum_{j \geq 0} a_{g,j} x^{q_1/n_1 + q_2/n_1n_2 + \cdots + (q_g + j)/n_1 \cdots n_g},
    \end{split}
\end{equation*}
where \( a_{j, 0} \neq 0 \) if \( j \neq 0 \). In the sequel, it will be convenient to consider enlarged sets of characteristic exponents containing other non-characteristic exponents of a Puiseux series \eqref{puiseux-series}, possibly corresponding to terms with zero coefficient. We will call them extended sets of characteristic exponents of a branch \(C\). Similarly, from an extended set of characteristic exponents, one defines extended Puiseux pairs of \(C\) from the same rules above.

\subsection{The semigroup of a branch}

Let \( \mathscr{O}_f \) be the local ring of \( f \). A Puiseux parameterization of \( f \) gives an injection \( \mathscr{O}_f \hookrightarrow \mathbb{C} \{t\} \). If we denote the \( t \)-adic valuation of \( \mathscr{O}_f \) by \( v_f \), then the semigroup \( \Gamma \subseteq \mathbb{Z}_{\geq 0} \) associated with \( f \) is
\[ \Gamma = \{v_f(\bar{g}) \in \mathbb{Z}_{\geq 0} \ |\ \bar{g} \in \mathscr{O}_f \setminus \{0\}\}. \]
It is well-known that this \( t \)-adic valuation coincides with the valuation defined by the intersection multiplicity with \( C \), that is, \( v_f(\bar{g}) = \dim_{\mathbb{C}} \mathscr{O}_{\mathbb{C}^2, x}/(g, f) \), see \cite[\S 2.6]{Cas00}. Since \( f \) is irreducible there exists a minimum integer \( c \in \mathbb{Z}_{>0} \), the conductor of \( f \), such that \( (t^c) \cdot \mathbb{C}\{t\} \subseteq \mathscr{O}_f \). As a result, any integer in \( [c, \infty) \) belongs to \( \Gamma \). In other words, \( \mathbb{Z}_{\geq 0} \setminus \Gamma \) is finite. Therefore, we can always find a minimal generating set \( \langle \overline{\beta}_0, \overline{\beta}_1, \dots, \overline{\beta}_g \rangle \) of \( \Gamma \), that is, \( \overline{\beta}_0, \dots, \overline{\beta}_g \) are the minimal integers such that \(\overline{\beta}_i \not\in \langle \overline{\beta}_0, \overline{\beta}_1, \dots, \overline{\beta}_{i-1} \rangle \), for \(j = 1, \dots, g\), and with \( \overline{\beta}_0 < \overline{\beta}_1 < \cdots < \overline{\beta}_g \), \( \gcd(\overline{\beta}_0, \overline{\beta}_1, \dots, \overline{\beta}_g) = 1 \).

\jump

The semigroup generators can be computed from the Puiseux pairs by the following formula \cite[\S II.3]{Zar86},
\begin{equation} \label{eq-char2semi}
    \overline{\beta}_i = n_{i-1} \overline{\beta}_{i-1} - \beta_{i-1} + \beta_i = e_i (n_i n_{i-1} \overline{m}_{i-1} + q_i) , \quad i = 2, \dots, g,
\end{equation}
with \( \overline{\beta}_0 = n \), \( \overline{\beta}_1 = \beta_1 \). By \eqref{eq-char2semi}, \( \gcd(e_{i-1}, \overline{\beta}_i) = e_i \) and \( e_{i-1} \centernot\mid \overline{\beta}_i \). Similarly, as before, we define the sequence of integers \( \overline{m}_i = \overline{\beta}_i / e_i \) called the reduced semigroup generators. Notice that \(n_i\) and \(\overline{m}_i\) remain coprime. The conductor \( c(\Gamma)  \) of \( \Gamma \) can also be computed from the numerical data introduced above, see \cite[\S II.3]{Zar86},
\begin{equation} \label{conductor}
    c(\Gamma) = n_g \overline{\beta}_g - \beta_g - (n - 1).
\end{equation}

The main property of plane branch semigroups that we will use is the following.

\begin{lemma}[{\cite[Lemma 2.2.1]{TeiApp86}}] \label{semigroup-prop}
    If \( \Gamma = \langle \overline{\beta}_0, \overline{\beta}_1, \dots, \overline{\beta}_g \rangle \) is the semigroup of a plane branch, then one has
    \[ n_i \overline{\beta}_i \in \langle \overline{\beta}_0, \overline{\beta}_1, \dots, \overline{\beta}_{i-1} \rangle. \]
    In particular, any \( \gamma \in \Gamma \) has an expression \( \gamma = \gamma_0 \overline{\beta}_0 + \cdots + \gamma_g \overline{\beta}_{g} \) with \( \gamma_i < n_i, i = 1, \dots, g \).
\end{lemma}

When the characteristic exponents or, equivalently, the Puiseux pairs have been extended, all the above definitions are still valid except that one obtains generators of the semigroup that are not minimal. In this case, the generators of \(\Gamma\) obtained from non-characteristic exponents via \eqref{eq-char2semi} are redundant, and they will be called extended semigroup generators.

\subsection{Divisorial valuations} \label{valuations}

Let \( \varpi : (Y, E) \longrightarrow (\mathbb{C}^2, x)\) be a proper birational morphism and consider \( E_i \) an exceptional component of the morphism \( \varpi \). If \(C\) is a plane branch such that its strict transform through \( \varpi \) intersects \( E_i \) transversally at a general point the valuation \( v_f \) defined by \( f \) coincides with the divisorial valuation \( v_i : \mathscr{O}_{\mathbb{C}^2, x} \longrightarrow \mathbb{Z}_{\geq 0} \) defined by \( E_i \). Hence, studying exceptional divisorial valuations is equivalent to studying plane branch valuations. 

\jump

We will define next the concept of characteristic exponents and Puiseux pairs of a divisorial valuation \(v_i\). The characteristic exponents of a divisorial valuation \(v_i\) are the characteristic exponents of a branch \(C\) extended by the last exponent that is common to all Puiseux series of any such branch \(C\). The Puiseux pairs of \(v_i\) are defined from the characteristic exponents of \(v_i\) following \cref{sec-puiseux-pairs}.

\subsection{Maximal contact elements} \label{max-contact}

A set of elements \( f_0, \dots, f_{g} \in \mathscr{O}_{\mathbb{C}^2, x} \) such that \( v_i(f_j) = \overline{\beta}_j \), for \( j = 0, \dots, g \), is called a set of maximal contact elements for the plane branch \( C \) or, equivalently, for the divisorial valuation \( v_i: \mathscr{O}_{\mathbb{C}^2, x} \longrightarrow \mathbb{Z}_{\geq 0} \). Maximal contact elements can be constructed by truncating a Puiseux series of \( C \). If \( s(x) \) is the Puiseux series of \( C \) from \eqref{puiseux-series}, then \( f_j, j = 1, \dots, g, \) have Puiseux series
\begin{equation*}
    s_j(x) = \sum_{\substack{k \in (n_1 \cdots n_j)\\ 1 \leq k < n_2 \cdots n_j m_1}} \hspace{-10pt} a_k x^{k/n} + \cdots + \sum_{\substack{k \in (n_j)\\ n_j m_{j-1} \leq k < m_j}} \hspace{-10pt} a_k x^{k/n} + \sum_{k \geq m_j} a'_k x^{k/n},
\end{equation*}
with \( a'_k \in \mathbb{C} \) arbitrary. More generally, maximal contact elements for a divisorial valuation \(v_i: \mathscr{O}_{\mathbb{C}^2, x} \longrightarrow \mathbb{Z}_{\geq 0}\) can be chosen so that their strict transforms intersect either dead-end divisors of the dual graph \(\Delta_{\pi}\) or the strict transforms of the maximal contact elements coincide with arrowheads of \(\Delta_{\pi}\).

\section{Numerical data of divisorial valuations} \label{result-valuations}

The goal of this section is to study the relation between different types of numerical data of adjacent divisorial valuations. Throughout this section \(\pi : (Y, E) \longrightarrow (\mathbb{C}^2, x)\) will be a log resolution, not necessarily minimal, of a plane curve \(C\) defined by \( f \in \mathscr{O}_{\mathbb{C}^2, x}\).

\jump

Recall that the dual graph \( \Delta_\pi \) of the log resolution \( \pi \) has as vertices the exceptional divisors \( E_i, i \in T_e \), which we also denote by \( E_i \) for convenience. Each intersection between exceptional components \( E_i \) is indicated by joining the corresponding vertices with an edge. Finally, the irreducible components of the strict transform of \( C \) are denoted by an arrowhead rooted in the vertex corresponding to the exceptional divisor intersecting each component. We denote by \( \Delta^e_{\pi} \) the restriction of \( \Delta_\pi \) to \( T_e \), that is, we remove the arrowheads.

\subsection{Dual graph decorations} \label{dual-graph}

In the sequel it will be useful to use the language of Eisenbud-Neumann diagrams \cite{EN85} associated with the dual graph \( \Delta_\pi \), see also \cite{NV12}. In these diagrams, the edges are decorated according to the following rules. For each \( i \in T_e \), we attach an edge decoration \( a \) next to \( E_i \) along an edge \( e \) adjacent to \( E_i \) by setting \( a \) equal to the absolute value of the determinant of the intersection matrix of all exceptional divisors appearing in the subgraph \( \Delta_{\pi} \setminus \{E_i\} \) in the direction of the edge \( e \).

\jump

The first relevant property is the so-called edge determinant rule. Fix an edge \( e \) between to vertices \( E_i \) and \( E_j \). Let \( \alpha \) and \( \beta \) the decorations along \( e \) next to \( E_i \) and \( E_j \), respectively. Denote also by \( \alpha_1, \dots, \alpha_n \), respectively \( \beta_1, \dots, \beta_m \), the decorations next to \( E_i \), respectively \( E_j \), along the other edges leaving \( E_i \), respectively \( E_j \).

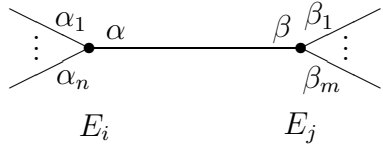
\begin{center}
    \begin{picture}(500,60)(50,-20)
        \put(240,20){\circle*{4}}
        \put(240,20){\line(1,0){80}}
        \put(320,20){\circle*{4}}
        \put(240,20){\line(-2,-1){30}}
        \put(240,20){\line(-2,1){30}}
        \put(220,23){\makebox(0,0){\(\vdots\)}}
        \put(320,20){\line(2,-1){30}} \put(320,20){\line(2,1){30}}
        \put(337,23){\makebox(0,0){\(\vdots\)}}
        \put(242,-10){\makebox(0,0){\(E_i\)}}
        \put(320,-10){\makebox(0,0){\(E_j\)}}
        \put(250,26){\makebox(0,0){\(\alpha\)}} \put(313,27){\makebox(0,0){\(\beta\)}}
        \put(327,32) {\makebox(0,0){\(\beta_1\)}}
        \put(328,8){\makebox(0,0){\(\beta_m\)}}
        \put(233,30) {\makebox(0,0){\(\alpha_1\)}}
        \put(234,8){\makebox(0,0){\(\alpha_n\)}}
    \end{picture}
\end{center}

Then, the edge determinant rule is
\begin{equation} \label{eq-edgedet}
     \alpha \beta - \prod_{k = 1}^n \alpha_k \prod_{k = 1}^m \beta_k = 1,
\end{equation}
with the usual convention that the empty product equals 1. The second property is that decorations are positive pair-wise coprime integers with at most two of them being greater than 1.

\subsection{Diagram calculus}

The second property related to Eisenbud-Neumann diagrams that will be useful is the so-called diagram calculus that computes the numerical data \( (N_i, k_i)_{i \in T} \) of the log resolution \( \pi \) from the decorations.

\begin{proposition}[{\cite{EN85, NV12}}] \label{prop-calculus}
    For any pair \( E_i, E_j, i, j \in T_e, \) of exceptional divisors, we denote by \( \ell_{ij} \) the product of all the decorations that are adjacent to, but not on, the path in \( \Delta_{\pi} \) joining \( E_i \) and \( E_j \). Then,
    \begin{equation} \label{eq-prop-calculus}
        N_i = \sum_{k \in T_s} \ell_{ik} N_k, \qquad \textnormal{and} \qquad k_i = \sum_{k \in T_e} \ell_{ik} (2 - \delta_k),
    \end{equation}
    where \(\delta_k\) is the valency of \(E_k\) in \(\Delta^e_{\pi}\), that is, the number of exceptional components intersecting \(E_i\).
\end{proposition}

More generally, we notice that the same arguments in the proof of \cref{prop-calculus} show that diagram calculus can also be used to compute the valuations \(v_i(g)\), for all \(i \in T_e\) and \(g \in \mathscr{O}_{\mathbb{C}^2, x}\), assuming the minimal log resolution of \(\{g = 0\}\) is dominated by \(\pi\). Namely, let \(\Delta_\pi(g)\) be the graph obtained from \(\Delta_\pi^e\) by adding the arrowheads corresponding to the strict transform components of \(\pi^* g\), and denote them by \(T_s(g)\), then
\begin{equation} \label{eq-calculus-valuations}
    v_i(g) = \sum_{k \in T_s(g)}\ell_{ik} v_k(g).
\end{equation}
In particular, one recovers the left-hand side of \eqref{eq-prop-calculus} by letting \(N_i = v_i(f)\), \(i \in T_e\).

\begin{lemma} \label{lemma-valuations-1}
    With the notations introduced above, for any pair of adjacent exceptional divisors \( E_i \) and \( E_j \) there exist \( \gamma_0, \gamma_1 \in \mathbb{Z}_{\geq 0} \) such that
    \[ N_i = \alpha \gamma_0 + \alpha_1 \cdots \alpha_n \gamma_1, \qquad \textnormal{and} \qquad N_j = \beta_1\cdots \beta_m \gamma_0 + \beta \gamma_1. \]
    Moreover, \(\gamma_0 = 0\), resp. \(\gamma_1 = 0\), if and only if there are no arrowheads in the connected component of \(\Delta_\pi \setminus \{E_j\}\), resp. \(\Delta_\pi \setminus \{E_i\}\), containing \(E_i\), resp. \(E_j\).
\end{lemma}
\begin{proof}
    Using \cref{prop-calculus}, we can decompose the expressions for \( N_i \) and \( N_j \) as
    \[ N_i = \sum_{k \in T_s^{(i)}} \ell_{ik} N_k + \sum_{k \in T_s^{(j)}} \ell_{ik} N_k, \quad \textnormal{and} \quad N_j = \sum_{k \in T_s^{(i)}} \ell_{jk} N_k + \sum_{k \in T_s^{(j)}} \ell_{jk} N_k, \]
    where \( T_s^{(i)} \), respectively \( T_s^{(j)} \), is the subset of \( T_s \) corresponding to the subgraph of \( \Delta_{\pi} \) rooted at \( E_i \), respectively \( E_j \), that arises after deleting the edge joining \( E_i \) and \( E_j \). Now, for all \( k \in T_s^{(i)} \), we have that \( \ell_{ik} \) is divisible by \( \alpha \), and for all \( k \in T_s^{(j)} \), \(\ell_{ik} \) is divisible by \( \alpha_1 \cdots \alpha_n \). Similarly, \( \ell_{jk} \) is divisible by \( \beta_1 \cdots \beta_m \) for \( k \in T_s^{(i)} \), and by \( \beta \) for \( k \in T_s^{(j)} \). By definition of \( \ell_{ik}, \ell_{jk} \),
    \[ \alpha^{-1} \sum_{k \in T_s^{(i)}} \ell_{ik} N_k = (\beta_1 \cdots \beta_m)^{-1} \sum_{k \in T_s^{(i)}} \ell_{jk} N_k, \quad \textnormal{and} \quad (\alpha_1 \cdots \alpha_m)^{-1} \sum_{k \in T_s^{(j)}} \ell_{ik} N_k = \beta^{-1} \sum_{k \in T_s^{(j)}} \ell_{jk} N_k. \]
    Hence, the lemma follows by letting \( \gamma_0, \gamma_1 \in \mathbb{Z}_{\geq 0} \) be the two quantities above. Finally, the summation defining \(\gamma_0\), respectively \(\gamma_1\), consists only of positive terms, so it is zero if and only if \(T_s^{(j)}\) is empty, respectively \(T_s^{(i)}\) is empty.
\end{proof}

We are indebted to Wim Veys for providing the proof of \cref{lemma-valuations-1}.

\subsection{Splice diagrams}

Let us introduce now the concept of splice diagrams from \cite{EN85}. For the particular case of algebraic links defined by plane curve singularities, the splice diagrams are constructed from Puiseux pairs according to the following rules. To one Puiseux pair \( (q_i, n_i) \) we attach the following decorated diagram,

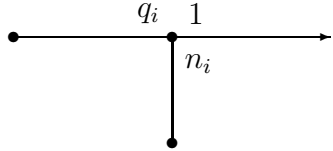
\begin{center}
    \begin{picture}(500,60)(50,-30)
        \put(200,20){\circle*{4}}
        \put(200,20){\line(1,0){60}}
        \put(260,20){\circle*{4}}
        \put(260,20){\vector(1,0){60}}
        \put(260,20){\line(0,-1){40}}
        \put(260,-20){\circle*{4}}

        \put(269,29){\makebox(0,0){\(1\)}}
        \put(251,29){\makebox(0,0){\(q_i\)}}
        \put(270,10){\makebox(0,0){\(n_i\)}}
    \end{picture}
\end{center}

For a plane branch with Puiseux pairs \( (q_1, n_1), \dots, (q_g, n_g) \), the diagrams for each pair are concatenated in the following way,

\begin{center}
    \begin{picture}(500,60)(50,-30)
        \put(100,20){\circle*{4}}
        \put(100,20){\line(1,0){60}}
        \put(160,20){\circle*{4}}

        \put(169,29){\makebox(0,0){\(1\)}}
        \put(151,29){\makebox(0,0){\(q_1\)}}
        \put(170,10){\makebox(0,0){\(n_1\)}}

        \put(160,20){\line(1,0){60}}
        \put(160,20){\line(0,-1){40}}
        \put(160,-20){\circle*{4}}
        \put(240,20){\makebox(-18,0){\(\cdots\)}}
        \put(242,20){\line(1,0){60}}
        \put(302,20){\circle*{4}}

        \put(312,29){\makebox(0,0){\(1\)}}
        \put(290,29){\makebox(0,0){\(\overline{m}_{g-1}\)}}
        \put(317,10){\makebox(0,0){\(n_{g-1}\)}}

        \put(302,20){\line(0,-1){40}}
        \put(302,-20){\circle*{4}}
        \put(302,20){\line(1,0){60}}
        \put(362,20){\circle*{4}}

        \put(370,29){\makebox(0,0){\(1\)}}
        \put(351,29){\makebox(0,0){\(\overline{m}_{g}\)}}
        \put(372,10){\makebox(0,0){\(n_{g}\)}}

        \put(362,20){\line(0,-1){40}}
        \put(362,-20){\circle*{4}}
        \put(362,20){\vector(1,0){60}}
    \end{picture}
\end{center}

where \( \overline{m}_i, i = 1, \dots, g \), are defined in \cref{semigroup}, see \cite[Prop. 1A.1]{EN85}. For non-irreducible curves, it is enough to understand the case of two branches. Assume the branches have Puiseux pairs \( (q_1, n_1), \dots, (q_r, n_r) \) and \( (q'_1, n'_1), \dots, (q'_s, n'_s) \), and that they intersect up to the \( i \)--th pair, that is, \( (q_1, n_1) = (q'_1, n'_1), \dots, (q_i, n_i) = (q'_i, n'_i) \). Then, the diagrams of the two branches are merged by one of the following configurations

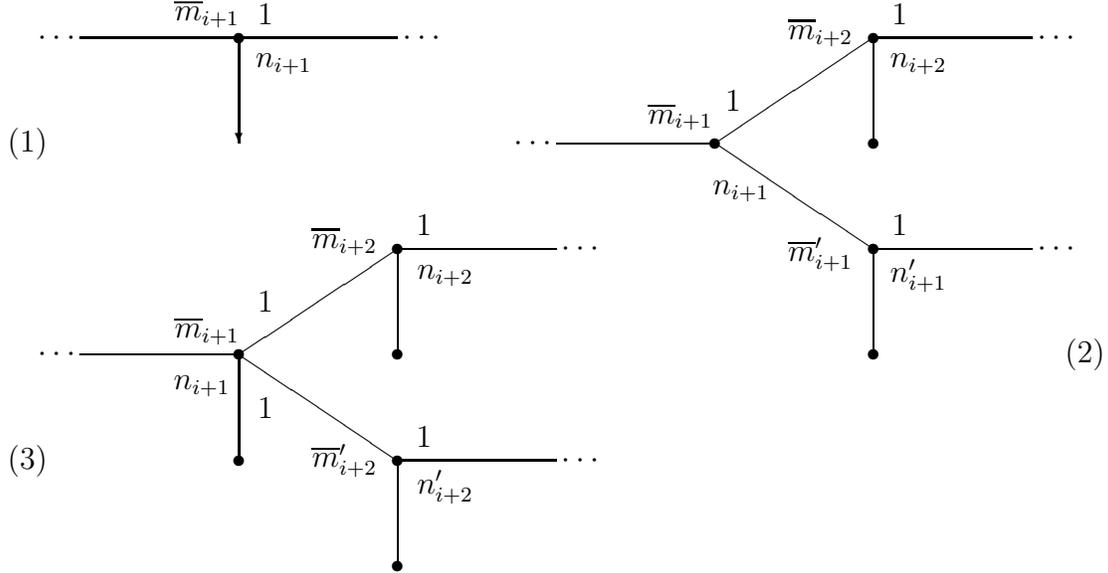
\begin{center}
    \begin{picture}(500,225)(50,-210)


        \put(100,0){\makebox(-18,0){\(\cdots\)}}
        \put(100,0){\line(1,0){60}}
        \put(160,0){\circle*{4}}
        \put(160,0){\line(1,0){60}}
        \put(220,0){\makebox(18,0){\(\cdots\)}}
        \put(160,0){\vector(0,-1){40}}

        \put(170,9){\makebox(0,0){\(1\)}}
        \put(148,9){\makebox(0,0){\(\overline{m}_{i+1}\)}}
        \put(177,-10){\makebox(0,0){\(n_{i+1}\)}}

        \put(80,-40){\makebox(0,0){\((1)\)}}


        \put(280,-40){\makebox(-18,0){\(\cdots\)}}
        \put(280,-40){\line(1,0){60}}
        \put(340,-40){\circle*{4}}

        \put(347,-24){\makebox(0,0){\(1\)}}
        \put(350,-58){\makebox(0,0){\(n_{i+1}\)}}
        \put(327,-30){\makebox(0,0){\(\overline{m}_{i+1}\)}}

        \put(340,-40){\line(3,2){60}}
        \put(400,0){\circle*{4}}
        \put(400,0){\line(0,-1){40}}
        \put(400,0){\line(1,0){60}}
        \put(460,0){\makebox(18,0){\(\cdots\)}}
        \put(400,-40){\circle*{4}}

        \put(410,9){\makebox(0,0){\(1\)}}
        \put(417,-10){\makebox(0,0){\(n_{i+2}\)}}
        \put(380,2){\makebox(0,0){\(\overline{m}_{i+2}\)}}

        \put(340,-40){\line(3,-2){60}}
        \put(400,-80){\circle*{4}}
        \put(400,-80){\line(0,-1){40}}
        \put(400,-80){\line(1,0){60}}
        \put(460,-80){\makebox(18,0){\(\cdots\)}}
        \put(400,-120){\circle*{4}}

        \put(410,-71){\makebox(0,0){\(1\)}}
        \put(417,-90){\makebox(0,0){\(n'_{i+1}\)}}
        \put(380,-82){\makebox(0,0){\(\overline{m}'_{i+1}\)}}

        \put(480,-120){\makebox(0,0){\((2)\)}}


        \put(100,-120){\makebox(-18,0){\(\cdots\)}}
        \put(100,-120){\line(1,0){60}}
        \put(160,-120){\circle*{4}}

        \put(170,-100){\makebox(0,0){\(1\)}}
        \put(148,-111){\makebox(0,0){\(\overline{m}_{i+1}\)}}
        \put(146,-132){\makebox(0,0){\(n_{i+1}\)}}
        \put(170,-140){\makebox(0,0){\(1\)}}

        \put(160,-120){\line(0,-1){40}}
        \put(160,-160){\circle*{4}}
        \put(160,-120){\line(3,2){60}}
        \put(220,-80){\circle*{4}}
        \put(220,-80){\line(0,-1){40}}
        \put(220,-80){\line(1,0){60}}
        \put(280,-80){\makebox(18,0){\(\cdots\)}}

        \put(230,-71){\makebox(0,0){\(1\)}}
        \put(238,-90){\makebox(0,0){\(n_{i+2}\)}}
        \put(200,-77){\makebox(0,0){\(\overline{m}_{i+2}\)}}

        \put(220,-120){\circle*{4}}
        \put(160,-120){\line(3,-2){60}}
        \put(220,-160){\circle*{4}}
        \put(220,-160){\line(0,-1){40}}
        \put(220,-160){\line(1,0){60}}
        \put(280,-160){\makebox(18,0){\(\cdots\)}}
        \put(220,-200){\circle*{4}}

        \put(230,-151){\makebox(0,0){\(1\)}}
        \put(238,-170){\makebox(0,0){\(n'_{i+2}\)}}
        \put(200,-160){\makebox(0,0){\(\overline{m}'_{i+2}\)}}

        \put(80,-160){\makebox(0,0){\((3)\)}}
    \end{picture}
\end{center}
The first diagram covers the case \(s = i\). The second is the case \(q_{i+1}/n_{i+1} < q'_{i+1}/n'_{i+1}\). For the last diagram, \(q_{i+1}/n_{i+1} = q'_{i+1}/n'_{i+1}\) but the corresponding coefficients in the Puiseux series are different. 

\jump

By \cite[Thm. 20.1]{EN85}, the splice diagrams can be constructed from the decorated dual graphs by only keeping all the edges between the vertices of valency greater or equal to 3 or supporting arrowheads. In \cite{EN85} there are some extra signs that we are disregarding by taking the absolute value of the determinants when defining the decorations of the dual graphs.

\subsection{Neighboring divisorial valuations} \label{sec-neigh-valuations}
Let \(E_i, i \in T_e,\) be an exceptional divisor. The neighboring valuations to the valuation associated with \( E_i \) are those associated with other divisors \( E_{i_j}, i_j \in T, \) such that \( E_i \cap E_{i_j} \neq \emptyset \). A neighborhood of \(E_i\) in dual graph \(\Delta_\pi\) looks like

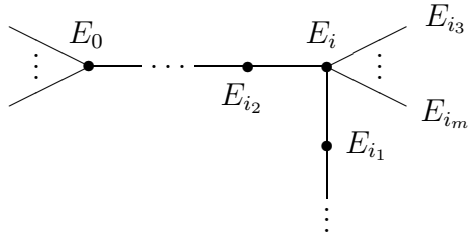
\begin{center}
    \begin{picture}(500,90)(0,-60)
        \put(160,23){\makebox(0,0){\(\vdots\)}}
        \put(180,20){\line(-2,-1){30}}
        \put(180,20){\line(-2,1){30}}
        \put(180,20){\circle*{4}}
        \put(180,20){\line(1,0){20}}
        \put(210,20){\makebox(0,0){\(\dots\)}}
        \put(220,20){\line(1,0){20}}
        \put(240,20){\circle*{4}}
        \put(240,20){\line(1,0){30}}
        \put(270,20){\circle*{4}}
        \put(270,20){\line(2,-1){30}}
        \put(270,20){\line(2,1){30}}
        \put(290,23){\makebox(0,0){\(\vdots\)}}

        \put(270,20){\line(0,-1){30}}
        \put(270,-10){\circle*{4}}
        \put(270,-10){\line(0,-1){20}}
        \put(270,-35){\makebox(0,0){\(\vdots\)}}

        \put(268,32){\makebox(0,0){\(E_i\)}}
        \put(238,08){\makebox(0,0){\(E_{i_2}\)}}
        \put(179,32){\makebox(0,0){\(E_0\)}}
        \put(285,-10){\makebox(0,0){\(E_{i_1}\)}}
        \put(315,38){\makebox(0,0){\(E_{i_3}\)}}
        \put(315,2){\makebox(0,0){\(E_{i_{m_i}}\)}}
    \end{picture}
\end{center}

Denote the divisors adjacent to \(E_i\) by \(E_{i_1}, E_{i_3}, \dots, E_{i_{m_i}}\), for a total of \(m_i \geq 3\) divisors. In the sequel, we will always use the convention that \(E_{i_1} \) and \(E_{i_2}\) are the exceptional divisors on the paths connecting \(E_i\) with divisors preceding it in the log resolution. Moreover, \(E_{i_2}\) will be on the path connecting \(E_i\) with the first exceptional divisor \(E_0\). It is possible that there is no divisor \(E_{i_1}\) in which case we will set \(E_{i_1} = 0\). The only case where there is no divisor \(E_{i_2}\), and we set \(E_{i_2} = 0\), is if \(E_i = E_0\). Notice that, in particular, \(m_i \geq r_i\).

\jump

Using the definitions from \cref{valuations}, assume that the valuation associated with \(E_i\) has \(g\) characteristic exponents or Puiseux pairs. The associated Puiseux pairs will be denoted by \((q_1, n_1), \dots, (q_g, n_g)\). We will next construct a set of extended characteristic exponents, or Puiseux pairs, for the neighboring valuations of \(E_i\).

\jump

The valuations of \(E_{i_j}, j = 1, \dots, m_i, \) share the first \(g-1\) characteristic exponents, or Puiseux pairs, with the valuation of \(E_i\). If the valuations of \(E_{i_1}\) or \(E_{i_2}\) have exactly \(g-1\) characteristic exponents, we will trivially extend its set of characteristic exponents with the exponent \(\frac{0}{1}\). Equivalently, in terms of Puiseux pairs, we will add the trivial pair \((0, 1)\). With this convention, the last pair in these extended sets of Puiseux pairs will always be denoted by \((b_g, a_g)\) and \((d_g, c_g)\) for \(E_{i_1}\) and \(E_{i_2}\), respectively.

\jump

Similarly, for any \(E_{i_j}, j \in \{3, \dots, m_i\}\), we will use the following set of extended characteristic exponents for their associated valuations. Take the set of \(g\) characteristic exponents of \(E_i\) and extend it with the last exponent of the set of characteristic exponents of the valuation of \(E_{i_j}\). In terms of Puiseux pairs, the last pair in this particular extended set for the valuation of \(E_{i_j}\) will be denoted by \((q_{g+1,j}, n_{g+1,j})\).

\jump

With these definitions, we have the following description of the decorations of the dual graph \(\Delta_\pi\) in a neighborhood of \(E_i\).

\begin{lemma} \label{lemma-valuations-2}
    The decorations from \cref{lemma-valuations-1} for \( E_{i_1} \) equal
    \[ \alpha = n_g, \qquad \alpha_1 \cdots \alpha_n = \overline{m}_g, \qquad \beta_1 \cdots \beta_m = a_g, \qquad \beta = a_g n_{g-1} \overline{m}_{g-1} + b_g. \]
    Similarly, for \( E_{i_2} \) the decorations are
    \[ \alpha = \overline{m}_g, \qquad \alpha_1 \cdots \alpha_n = n_g, \qquad \beta_1 \cdots \beta_m = c_g n_{g-1} \overline{m}_{g-1} + d_g, \qquad \beta = c_g. \]
    Finally, for any \(E_{i_j}, j \in \{3, \dots, m_i\}\), the decorations are
    \[ \alpha = 1, \qquad \alpha_1 \cdots \alpha_n = n_g \overline{m}_g, \qquad \beta_1 \cdots \beta_m = n_{g+1,j}, \qquad \beta = n_{g+1,j} n_g \overline{m}_g + 1, \]
    with \(q_{g+1,j} = 1\).
\end{lemma}
\begin{proof}
    By the previous discussion the decorations surrounding \( E_i \) in the dual graph that are different from \( 1 \) are precisely \( n_g \) and \( \overline{m}_g \). We will do the case of \(E_{i_1}\), the case for \(E_{i_2}\) being similar. Notice that attaching an arrowhead to \(E_{i_1}\) in the dual graph \(\Delta_{\pi}\) does not change the decorations. Moreover, by computing the resulting splice diagram one obtains the configuration in the second case above. Hence, the result follows from \eqref{eq-char2semi}.
    
    \jump
    
    For the case of a divisor \(E_{i_j}, j \in \{3, \dots, m_i\}\), attaching an arrowhead to \(E_{i_j}\) in \(\Delta_\pi\) does not change the decorations. Computing the splice diagrams, one obtains the configuration in the first case above. Hence, \(\alpha = 1\) and the edge determinant rule gives \(\beta = \beta_1 \cdots \beta_m \overline{m}_g n_g + 1\) with \(\beta_1 \cdots \beta_m = n_{g+1,j}\). It then follows from \eqref{eq-char2semi} that \(q_{g+1,j} = 1\).
\end{proof}

The previous results allow us to express the numerical data for \(E_i\) and its neighboring divisors in terms of the numerical data of a branch. 

\begin{lemma} \label{lemma-valuations-3}
Let \(C_i\) be a plane branch whose strict transform intersects an exceptional divisor \(E_i, i \in T_e, \) at a general point. Then, 
\[
    N_i(C_i) = n_g \overline{m}_g, \quad N_{i_1}(C_i) = a_g \overline{m}_g, \quad N_{i_2}(C_i) = n_g(c_g n_{g-1} \overline{m}_{g-1} + d_g).
\]
Moreover,
\[
    k_i = n_1 \cdots n_g + m_g, \quad k_{i_1} = n_1 \cdots n_{g-1} a_g + a_g m_{g-1} + b_g, \quad k_{i_2} = n_1 \cdots n_{g-1} c_g + c_g m_{g-1} + d_g,
\]
and \[k_{i_j} = n_1 \cdots n_{g} n_{g+1,j} + n_{g+1,j} m_{g} + 1, \qquad \textnormal{for} \qquad j = 3, \dots, m_i.\]
\end{lemma}
\begin{proof}
Since there is only one arrowhead attached to \(E_{i}\), using \cref{prop-calculus,lemma-valuations-2} one obtains the formulas for \(N_i(C_i), N_{i_1}(C_i)\), and \(N_{i_2}(C_i)\). For the second part of the lemma, it is enough to assume that \(\pi\) is the minimal resolution of \(C_i\) plus, possibly, some extra blow-ups performed on \(C_i\) to reach the exceptional divisor \(E_i\). Then, since \(\delta_i = -1, 0, 1\), for all \(i \in T_e\), \cref{prop-calculus} implies
\[
 k_i = \overline{m}_g + \sum_{j=0}^{g-1} n_g \overline{m}_j - \sum_{j=1}^g n_g n_{j} \overline{m}_j = n_1\cdots n_g + m_g,
\]
where the last equality follows from applying \eqref{eq-char2semi} recursively. The formulas for \(k_{i_j}, j = 1, \dots, m_i \), follow from the one for \(k_i\), \cref{eq-char2semi} and the definition of the extended Puiseux pairs for the divisorial valuations associated with \(E_{i_j}\).

\end{proof}

\section{Proof of Theorem A} \label{section-eps}

In this section, we will work with a fixed exceptional divisor \( E_i, i \in T_e, \) coming from the minimal log resolution \( \pi \) of an arbitrary plane curve \(C\) defined by \(f : (\mathbb{C}^2, x) \longrightarrow (\mathbb{C}, 0)\). We keep using the notations and conventions from previous sections. Let \(\omega \in \Omega^2_{\mathbb{C}^2, x}\) be a top differential form. The residue numbers \(\epsilon_{i_j}(\omega), j = 1, \dots, m_i,\) at \(E_i\) satisfy the following relation. This result is a generalization of \cite[Prop. II.3.1]{Loe88}, see also \cite{Vey91}, which only covers the standard volume form case.

\begin{lemma} \label{lemma-sumeps}
    Let \( E_i \) be an exceptional divisor and \(\omega \in \Omega^2_{\mathbb{C}^2, x}\) such that \(F_\pi + \textnormal{Div}(\pi^* \omega)\) is an SNC divisor. Then, 
    \[ \epsilon_{i_1}(\omega) + \cdots + \epsilon_{i_{m_i}}(\omega) + \delta_{i_1}(\omega) + \cdots + \delta_{i_p}(\omega) = -2 + m_i, \]
    where \( \delta_{i_1}(\omega),\dots , \delta_{i_p}(\omega) \) are the multiplicities of the irreducible components of \( \textnormal{Div}(\pi^* \omega) \) not in \(\textnormal{Supp}(F_\pi)\) that intersect \(E_i\).
\end{lemma}

\begin{proof}
    Consider the \( \mathbb{Q} \)-divisor \( \sigma_{i}(\omega) F_{\pi} + \textnormal{Div}(\pi^* \omega) \) and compute its intersection multiplicity with the divisor \( E_i \). By the adjunction formula for surfaces, this intersection number is equal to \( -E_i^2 - 2 \). On the other hand,
    \[ (\sigma_{i}(\omega) F_{\pi} + \textnormal{Div}(\pi^* \omega))\cdot E_i = \sum_{j=1}^{m_i} (\epsilon_{i_j}(\omega) - 1) + \sum_{j=1}^p \delta_{i_j}(\omega) - E_i^2. \]
    Combining both computations, and using that \(\epsilon_{i_j}(\omega) = 1\) if \(E_{i_j}\) is zero, gives the result.
\end{proof}

The differential forms \(\omega \in \Omega^2_{\mathbb{C}^2, x}\) that will be considered throughout the rest of this work will be as follows. For each vertex \(E_j, j\in T,\) of the dual graph \(\Delta_\pi\), let \(f_j \in \mathscr{O}_{\mathbb{C}^2, x}\) be an equation of a branch \(C_j\) whose strict transform intersects the divisor \(E_j\) at a general point. When \(j \in T_s\), that is, \(E_j\) is an arrowhead, this just means that \(f_j\) is a branch of \(f\). Then,
\[ \omega = g \dd x, \qquad g = \prod_{r_j \neq 2} f_j^{\gamma_j}, \qquad \gamma_j \in \mathbb{Z}_{\geq 0},\]
where the product ranges over all divisors \(E_j\) with valency different from \(2\) in the dual graph \(\Delta_\pi\). After \cref{max-contact}, this set of elements contains a set of maximal contact elements for any divisorial valuation \(E_k, k \in T_e\). Moreover, all these differential forms satisfy \cref{lemma-sumeps} with \(\delta_{i_1}(\omega) = \gamma_i, \) if \(r_i \neq 2\), and \(\delta_{i_1}(\omega) = 0\) otherwise.

\subsection{Formulae for the residue numbers}
The goal of this section is to give closed formulas for the residue numbers at \(E_i\) in terms of the numerical data of the associated valuation.

\begin{lemma} \label{lemma-linear-valuations}
Let \(E_{i_1}, \dots, E_{i_{m_i}}\) be the divisors adjacent to \(E_i\). Then, there exist \(\beta_1, \dots, \beta_{m_i} \in \mathbb{Z}_{\geq 0}\) such that
\[
N_{i_1} = \frac{a_g}{n_g} N_i + \frac{\beta_1}{n_g}, \quad N_{i_2} = \frac{c_g n_{g-1} \overline{m}_{g-1} + d_g}{\overline{m}_g} N_i + \frac{\beta_2}{\overline{m}_g}, \quad N_{i_j} = n_{g+1,j} N_i + \beta_{j},\ j = 3, \dots, m_i,
\]
where \(\beta_j\), for \(j = 1, \dots, m_i, \) is zero if and only if there are no arrowheads in the connected component of \(\Delta_\pi \setminus \{E_i\}\) containing \(E_{i_j}\). Moreover, 
\[
    k_{i_1} = \frac{a_g}{n_g} k_i + \frac{1}{n_g}, \qquad k_{i_2} = \frac{c_g}{n_g} k_i + \frac{1}{n_g}, \qquad k_{i_j} = n_{g+1,j} k_i + 1,\ j = 3, \dots, m_i.
\]
\end{lemma}
\begin{proof}
The proof follows from \cref{lemma-valuations-1,lemma-valuations-2}. We will do the case of \(E_{i_1}\) since the other cases are similar. From these two lemmas,
\[
    N_i = n_g \beta'_1 + \overline{m}_g \beta_1, \qquad N_{i_1} = a_g \beta'_1 + (a_g n_{g-1} \overline{m}_{g-1} + b_g) \beta_1.
\]
Solving for \(\beta'_1\) in the expression for \(N_i\), substituting in the expression for \(N_{i_1}\) and using the edge determinant rule gives the result. \cref{lemma-valuations-1} implies that \(\beta_1\) is zero if and only if there are now arrowheads in the connected component of \(\Delta_\pi \setminus \{E_i\}\) containing \(E_{i_1}\).

\jump

The expressions in the second part of the lemma follow from a direct computation using \cref{lemma-valuations-3}, \cref{eq-char2puiseux}, and the edge determinant rule.
\end{proof}

\begin{corollary} \label{cor-valuations-3}
Let \(g \in \mathscr{O}_{\mathbb{C}^2, x}\) be a germ of a holomorphic function. Then, there exist \(\alpha_1, \dots, \alpha_{m_i+1} \in \mathbb{Z}_{\geq 0}\) depending only on \(g\) such that
\begin{equation} \label{eq-cor-valuations-3}
    v_i(g) = \overline{m}_g \alpha_1 + n_g \alpha_2 + \sum_{j = 3}^{m_i+1} n_g \overline{m}_g \alpha_j,
\end{equation}
where \(\alpha_j\), for \(j = 1, \dots, m_i,\) is zero if and only if there are no arrowheads in the connected component of \(\Delta_\pi(g) \setminus \{E_i\}\) containing \(E_{i_j}\). Moreover,
\[
v_{i_1}(g) = \frac{a_g}{n_g} v_i(g) + \frac{\alpha_1}{n_g}, \quad v_{i_2}(g) = \frac{c_g n_{g-1} \overline{m}_{g-1} + d_g}{\overline{m}_g} v_i(g) + \frac{\alpha_2}{\overline{m}_g}, \quad v_{i_j}(g) = n_{g+1, j} v_{i}(g) + \alpha_j,
\]
for \(j = 3, \dots, m_i\), and \(\alpha_{m_i+1} = \gamma_i\).
\end{corollary}
\begin{proof}
    These expressions follow from the same arguments used in proofs of \cref{lemma-valuations-1,lemma-linear-valuations} but using \eqref{eq-calculus-valuations} instead of \eqref{eq-prop-calculus}.
\end{proof}

Using the previous results we can express the residue numbers at \(E_i\) in terms of the numerical data of the valuation associated with \(E_i\).

\begin{proposition} \label{prop-eps-rup}
Let \(\omega = g \dd x \in \Omega^2_{\mathbb{C}^2, x}\) be a top differential form. Using the notations from \cref{lemma-valuations-3,cor-valuations-3},
\[
    \epsilon_{i_1}(\omega) = \frac{\alpha_1 + 1}{n_g} + \frac{\beta_1}{n_g} \sigma_i(\omega), \quad \epsilon_{i_2}(\omega) = \frac{\alpha_2 + m_{g-1} + n_1 \cdots n_{g-1} - n_{g-1} \overline{m}_{g-1}}{\overline{m}_g} + \frac{\beta_2}{\overline{m}_g}\sigma_i(\omega),
\]
and
\[
    \epsilon_{i_j}(\omega) = \alpha_j + 1 + \beta_j \sigma_i(\omega), \quad \textnormal{for} \quad j = 3, \dots, m_i.
\]
\end{proposition}
\begin{proof}
It is enough to replace the expressions in \cref{lemma-valuations-3,cor-valuations-3} in the definition \eqref{eq-definition-eps} of residue numbers at \(E_i\).
\end{proof}

\begin{corollary} \label{cor-eps-rup}
    One has the following inequalities,
    \begin{equation}
        \begin{split}
            -\frac{m_g + n_1 \cdots n_g + v_i(g)}{n_g \overline{m}_g} + \frac{\alpha_1 + 1}{n_g} & \leq  \epsilon_{i_1}(\omega) \leq \frac{\alpha_1 + 1}{n_g}, \\
            -\frac{\alpha_1 + 1}{n_g} & \leq \epsilon_{i_2}(\omega) \leq \frac{m_g + n_1 \cdots n_g + v_i(g)}{n_g \overline{m}_g} - \frac{\alpha_1 + 1}{n_g}. \\
        \end{split}
    \end{equation}
    The lower bounds are strict if \(r_i \geq 3\). The upper bounds are strict if there exists an arrowhead in the connected component of \(\Delta_\pi \setminus \{E_i\}\) containing \(E_{i_1}\), resp. \(E_{i_2}\). In addition,
    \[
            -\frac{m_g + n_1 \cdots n_g + v_i(g)}{n_g \overline{m}_g} + {\alpha_j + 1} \leq \epsilon_{i_j}(\omega) < {\alpha_j + 1}, \qquad \qquad \qquad \qquad \qquad \\
    \]
    for \(j = 3, \dots, m_i\).
\end{corollary}
\begin{proof}
    Start by replacing the following equality obtained from \eqref{eq-cor-valuations-3}
    \[ \alpha_2 = \frac{v_i(g) - \overline{m}_g \alpha_1 - n_g\overline{m}_g \alpha}{n_g}, \] 
    where \(\alpha = \alpha_3 + \cdots + \alpha_{m_i+1}\), in the expression for \( \epsilon_{i_2}(\omega) \). Then, the upper bounds follow from \( \beta_1, \dots, \beta_{m_i}, \alpha \geq 0 \). The strict lower bounds follow from \cref{lemma-linear-valuations}. Moreover, since the log resolution is minimal, \(\beta_j > 0\), for all \(j = 3, \dots, m_i\). For the lower bounds, recall that from \cref{lemma-valuations-1}
    \[ N_i = n_g \beta'_1 + \overline{m}_g \beta_1 = \overline{m}_g \beta'_2 + n_g \beta_2 = \beta'_j + n_g \overline{m}_g \beta_j, \]
    for \(j = 3, \dots, m_i\). Solving for \( \beta_j, j = 1, \dots, m_i, \) and replacing the respective expressions in \cref{prop-eps-rup} yields the lower bounds, since \(\beta'_1, \dots, \beta'_{m_i} \geq 0\). Finally, \(r_i \geq 3\) implies that there is at least one arrowhead in connected components of \(\Delta_\pi \setminus \{E_{i_1}\}\) and \(\Delta_\pi \setminus \{E_{i_2}\}\) containing \(E_i\) since the log resolution is minimal. Thus, \(\beta'_1, \beta'_2 > 0\). 
\end{proof}

\subsection{Bounds for the residue numbers} \label{sec-bounds-residue}
We keep the same notations and conventions as in the previous section. If \(E_i, i \in T_e, \) is an exceptional divisor, we will denote by \(\Gamma_i\) the semigroup of the valuation associated with \(E_i\) and by \(f_{j_0}, \dots, f_{j_g} \in \mathscr{O}_{\mathbb{C}^2, x}, j_0, \dots, j_g \in T, \) a set of maximal contact elements for the valuation of \(E_i\).

\jump

With the conventions set in \cref{sec-neigh-valuations}, the arrowheads of the curves defined by \(f_{j_0}, \dots, f_{j_g}\) are always in the direction of \(E_{i_1}\) and \(E_{i_2}\). If \(E_{i_1} \neq 0\), the arrowhead of \(f_{j_g}\) is in the direction of \(E_{i_1}\) and all the rest are in the direction of \(E_{i_2}\). Otherwise, all arrowheads appear in the direction of \(E_{i_2}\). 

\begin{proposition} \label{prop-rup2}
    Let \(E_i, i \in T_e,\) be an exceptional divisor. Given any differential form \(\omega \in \Omega^2_{\mathbb{C}^2, x}\), there always exists \(\omega' \in \Omega^2_{\mathbb{C}^2, x}\) such that
    \[ -1 \leq \epsilon_{i_1}(\omega'), \epsilon_{i_2}(\omega') \leq 1, \] 
    and with \(\sigma_i(\omega') = \sigma_i(\omega)\). Moreover, the lower bounds are strict if \(r_i \geq 3\). If there is an arrowhead in the connected component of \(\Delta_\pi \setminus \{E_i\}\) containing \(E_{i_1}\), resp. \(E_{i_2}\), the upper bound for \(\epsilon_{i_1}(\omega')\), resp. \(\epsilon_{i_2}(\omega')\), is strict.
\end{proposition}
\begin{proof}
    From \cref{cor-valuations-3} we see that if \(f_v, v \in T\), has an arrowhead in the direction of \(E_{i_j}, j \in \{3, \dots, m_i\}\), then \(v_i(f_v) = n_g \overline{m}_g \alpha_v \), for some \(\alpha_v \in \mathbb{Z}_{> 0}\). Therefore, \cref{semigroup-prop} implies that \(f_v\) can be replaced by a product of maximal contact elements for the valuation of \(E_i\) while \(v_i(g)\), and hence \(\sigma_i(\omega)\), remains constant. Therefore, we can always assume that \(\omega\) is such that \(\alpha_3, \dots, \alpha_{m_i+1} = 0\).

    \jump

    Assume now that \(\omega\) is such that the lower bounds are not satisfied, that is, \(\epsilon_{i_1}(\omega) < -1\) and \(\epsilon_{i_2}(\omega) < -1\). We observe now that both inequalities cannot happen at the same time. Indeed, if \(r_i > 1\), \cref{lemma-sumeps,prop-eps-rup} imply
    \[
        \epsilon_{i_1}(\omega) + \epsilon_{i_2}(\omega) = -\sum_{j = 3}^m \beta_j \sigma_i(\omega) \geq 0,
    \]
    so \(\epsilon_{i_1}(\omega), \epsilon_{i_2}(\omega) < 0\) is not possible. If \(r_i = 1\), then \(\epsilon_{i_1}(\omega) = 1\) and \(\epsilon_{i_2}(\omega) \leq -1\). Let us then start assuming that \(\epsilon_{i_1}(\omega) < -1\). Combining this assumption with the lower bound from \cref{cor-eps-rup} and \eqref{conductor} one obtains
    \[ 
        v_i(g) - (\alpha_1 + 1) \overline{m}_g > n_g \overline{m}_g - m_g - n_1 \cdots n_g = c(\Gamma_i) - 1.
    \]
    Thus, \( v_i(g) - (\alpha_1 + 1) \overline{m}_g \in \Gamma_i \) and denote \( \gamma_0, \dots, \gamma_g \in \mathbb{Z}_{\geq 0} \) its coordinates in \( \Gamma_i \) fulfilling \cref{semigroup-prop}. We claim now that \( \gamma_g = n_g - 1 \). To see this, first observe that \( v_i(g) \equiv \alpha_1 \overline{m}_g \pmod{n_g} \). Consequently, \[ v_i(g) - (\alpha_1 + 1) \overline{m}_g \equiv - \overline{m}_g \equiv \gamma_g \overline{m}_g \pmod{n_g}. \]
    From the fact that \( \overline{m}_g \) and \( n_g \) are coprime, it follows that \( \gamma_g \equiv -1 \pmod{n_g} \) and since \( 0 \leq \gamma_g \leq n_g - 1 \) the claim follows. Then, define
    \[
        \omega' = g' \dd x, \qquad \textnormal{with} \qquad g' = f_{j_0}^{\gamma_0} f_{j_1}^{\gamma_1} \cdots f_{j_{g-1}}^{\gamma_{g-1}} f_{j_g}^{\alpha_1 + n_g}.
    \]
    Several checks are in order. First, let us check that \(v_i(g') = v_i(g)\), and thus \(\sigma_i(\omega') = \sigma_i(\omega)\). Indeed,
    \begin{displaymath}
        \begin{split}
            v_i(g') & = \gamma_0 v_i(f_{j_0}) + \cdots + \gamma_{g-1} v_i(f_{j_{g-1}}) + (\alpha_1 + n_g) v_i(f_{j_g}) \\
            & = \gamma_0 v_i(f_{j_0}) + \cdots + \gamma_{g} v_i(f_{j_g}) + (\alpha_1 + 1) v_i(f_{j_g}) \\
            & = v_i(g),
        \end{split}
    \end{displaymath}
    since \( v_i(f_{j_g}) = \overline{m}_g \) and \( \gamma_{g} = n_g - 1\). Secondly, notice that \(\epsilon_{i_1}(\omega') = \epsilon_{i_1}(\omega) + 1 \) by \cref{prop-eps-rup}. Therefore, either \( -1 \leq \epsilon_{i_1}(\omega') \) or we can replace \( \omega \) by \( \omega' \) and repeat the same process until \( -1 \leq \epsilon_{i_1}(\omega') \).
    
    \jump

    Let us assume now that \(\epsilon_{i_2}(\omega) < -1\). Using \cref{cor-eps-rup} one gets that \(\alpha_1 > n_g - 1\). Hence, we can apply \cref{semigroup-prop} and obtain \(\gamma_0, \dots, \gamma_g \in \mathbb{Z}_{\geq 0}\)
    \[
        \omega' = g' \dd x \qquad \textnormal{with} \qquad g' = f_{j_0}^{\gamma_0} f_{j_1}^{\gamma_1} \cdots f_{j_g}^{\gamma_g},
    \]
    such that \(\gamma_g \leq n_{g} - 1\) and \(v_i(g') = v_i(g)\), thus \(\sigma_i(\omega') = \sigma_i(\omega)\). It then follows from \cref{cor-eps-rup} that \(-1 \leq \epsilon_{i_1}(\omega)\). Notice that if \(r_i \geq 3\), the lower bounds from \cref{cor-eps-rup} are strict, and using the same arguments above we can make the lower bounds for \(\epsilon_{i_1}(\omega')\) and \(\epsilon_{i_2}(\omega')\) strict.

    \jump

    Let us focus now on the upper bounds. By the previous constructions, we can always assume that \(\omega = g \dd x \in \Omega^2_{\mathbb{C}^2, x}\) is such that \(-1 \leq \epsilon_{i_1}(\omega), \epsilon_{i_2}(\omega)\) and \(\alpha_3, \dots, \alpha_{m_i+1} = 0\). Notice that if \(r_i = 1, 2\), then the upper bounds are automatically satisfied by \cref{lemma-sumeps}. So let us assume that \(\epsilon_{i_1}(\omega),\epsilon_{i_2}(\omega) > 1\) and \(r_i \geq 3\).  Let \(\gamma \in \mathbb{Z}_{\geq 1}\) be the maximum integer such that \(\epsilon_{i_2}(\omega) > \gamma\), then the upper bound for \(\epsilon_{i_2}(\omega)\) in \cref{cor-eps-rup} gives
    \[
        v_i(g) - (\alpha_1 + 1)\overline{m}_g - n_g \overline{m}_g (\gamma - 1) > c(\Gamma_i) - 1.
    \]
    Proceeding similarly as above, let \(\gamma_1, \dots, \gamma_g \in \mathbb{Z}_{> 0}\) from \cref{semigroup-prop} and define
    \begin{equation} \label{eq-prop-rup2}
        \omega' = g' \dd x \qquad \textnormal{with} \qquad g' = f_{j_0}^{\gamma_0} \cdots f_{j_{g-1}}^{\gamma_{g-1}} f_{j_g}^{\alpha_1} f_{i}^\gamma,
    \end{equation}
    where \(f_{i} \) is such that \(v_i(f_{i}) = n_g \overline{m}_g\). As before, \(\gamma_g = n_g - 1\) and one checks that \(v_i(g') = v_i(g)\). Moreover, it follows from \cref{prop-eps-rup} that \(\epsilon_{i_1}(\omega') = \epsilon_{i_1}(\omega) \) and \(\epsilon_{i_2}(\omega') = \epsilon_{i_2}(\omega) - \gamma \) since 
    \[
        \alpha'_2 = \frac{v_i(g') - \overline{m}_g \alpha_1 - n_g \overline{m}_g \gamma}{n_g} = \alpha_2 - \overline{m}_g \gamma.
    \]
    Therefore, \(-1 \leq \epsilon_{i_2}(\omega') \leq 1\). Now, combining \(\epsilon_{i_1}(\omega') > 1\) with the upper bound in \cref{cor-eps-prop} yields \(\alpha_g > n_g - 1\), so we define
    \[
        \omega'' = g'' \dd x \qquad \textnormal{with} \qquad g'' = f_{j_0}^{\gamma_0} \cdots f_{j_{g-1}}^{\gamma_{g-1}} f_{j_g}^{\alpha_1\hspace{-6pt} \pmod{n_g}} f_{i}^{\gamma + \lfloor \alpha_1 / n_g \rfloor},
    \]
    and, since \(v_i(f_{3,v}) = n_g \overline{m}_g\), it is clear that \(v_i(g'') = v_i(g') = v_i(g)\). Finally, using \cref{prop-eps-rup} one can check that \(\epsilon_{i_2}(\omega'') = \epsilon_{i_2}(\omega') \) and \(-1 \leq \epsilon_{i_1}(\omega'') \leq 1\). 

    \jump
    
    If there is an arrowhead in the connected component of \(\Delta_\pi \setminus \{E_i\}\) containing \(E_{i_1}\), resp. \(E_{i_2}\), the respective upper bounds from \cref{cor-eps-rup} are strict, hence the same arguments above show that the upper bounds for \(\epsilon_{i_1}(\omega'')\) and \(\epsilon_{i_2}(\omega'')\) can be made strict.
\end{proof}

\begin{corollary} \label{cor2-prop-rup2}
    Assume \(r_i \geq 3\). If \(\epsilon_{i_1}(\omega) \in \mathbb{Z}\), respectively \(\epsilon_{i_2}(\omega) \in \mathbb{Z}\), then there exists \(\omega' \in \Omega^2_{\mathbb{C}^2, x}\) with \(\epsilon_{i_1}(\omega') = 0\), respectively \(\epsilon_{i_2}(\omega') = 0\), if and only if there is at least one arrowhead in the connected component of \(\Delta_\pi \setminus \{ E_i\}\) containing \( E_{i_1} \), respectively \(E_{i_2}\).
\end{corollary}
\begin{proof}
    The converse implication follows from \cref{prop-rup2}. For the direct implication, we will argue by contradiction. If there is no arrowhead in the respective connected component, \(\beta_1 = 0\), respectively \(\beta_2 = 0\), by \cref{lemma-linear-valuations}. Hence, if there exists an \(\omega'\) such that \(\epsilon_{i_1}(\omega') = 0\), resp. \(\epsilon_{i_2}(\omega') = 0\), \cref{prop-eps-rup} would imply
    \[
        \alpha'_1 + 1 = 0 \qquad \textnormal{resp.} \qquad m_{g-1} +  n_1 \cdots n_{g-1} - n_{g-1} \overline{m}_{g-1} + \alpha'_2 = 0,
    \]
    which is impossible since \(\alpha'_1, \alpha'_2 \geq 0\). Hence, \(\epsilon_{i_1}(\omega') \neq 0\), resp. \(\epsilon_{i_2}(\omega') \neq 0\).
\end{proof}

Assuming that \(r_i \geq 3\), it is always possible to find elements \(f_v \in \mathscr{O}_{\mathbb{C}^2,x}, v \in T, r_v = 1,\) such that the arrowhead of the branch \(f_v = 0\) appears in the direction of a chosen divisor \(E_{i_3}, \dots, E_{i_{m_i}}\) and, moreover, \(v_i(f_v) = N_i(C_i) = n_g \overline{m}_g\). Denote by \(f_{v, j}\) any such element whose associated arrowheads appear in the direction of \(E_{j}, j = 3, \dots, m_i\).

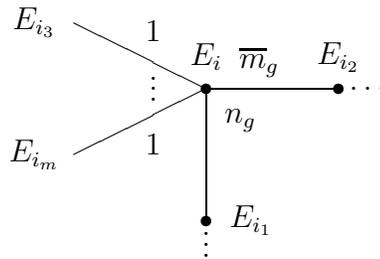
\begin{center}
    \begin{picture}(500,105)(-30,-55)
        \put(160,23){\makebox(0,0){\(\vdots\)}}
        \put(180,20){\line(-2,-1){50}}
        \put(180,20){\line(-2,1){50}}
        \put(180,20){\circle*{4}}
        \put(180,20){\line(1,0){50}}
        \put(180,20){\line(0,-1){50}}

        \put(180,-30){\circle*{4}}
        \put(180,-36){\makebox(0,0){\(\vdots\)}}

        \put(230,20){\circle*{4}}
        \put(241,20){\makebox(0,0){\(\dots\)}}

        \put(180,32){\makebox(0,0){\(E_i\)}}
        \put(230,32){\makebox(0,0){\(E_{i_2}\)}}
        \put(197,-30){\makebox(0,0){\(E_{i_1}\)}}

        \put(115,45){\makebox(0,0){\(E_{i_3}\)}}
        \put(115,-5){\makebox(0,0){\(E_{i_{m_i}}\)}}

        \put(193,8){\makebox(0,0){\(n_{g}\)}}
        \put(200,30){\makebox(0,0){\(\overline{m}_{g}\)}}
        \put(160,42){\makebox(0,0){\(1\)}}
        \put(160,-1){\makebox(0,0){\(1\)}}
    \end{picture}
\end{center}

\begin{proposition} \label{eps-proposition}
    Let \(\omega \in \Omega^2_{\mathbb{C}^2, x}\) be a form such that \(-1 \leq \epsilon_{i_1}(\omega), \epsilon_{i_2}(\omega) \leq 1\). Then, there exists \(\omega' \in \Omega^2_{\mathbb{C}^2, x}\) such that
    \[
        -1 \leq \epsilon_{i_j}(\omega') < 1, \quad \textnormal{for} \quad j = 3, \dots, m_i, 
    \] 
    with \(\sigma_i(\omega') = \sigma_i(\omega'), \epsilon_{i_1}(\omega') = \epsilon_{i_1}(\omega),\) and \(\epsilon_{i_2}(\omega) = \epsilon_{i_2}(\omega')\). Moreover, \(\sigma_{v}(\omega') = \sigma_{v}(\omega)\), for all \(v \in T\) in the connected components of \(\Delta_\pi \setminus \{E_i\}\) containing \(E_{i_1}\) or \(E_{i_2}\).
\end{proposition}
\begin{proof}
    If \(r_i = 1, 2 \) there is either nothing to prove or the bounds follow from \cref{lemma-sumeps}. Otherwise, arguing as in \cref{prop-rup2}, we can assume that all \(f_v, v \in T,\) in the direction of \(E_{i_j}, j \in \{3, \dots, m_i\}\), are replaced by copies of \(f_i\). Notice that, by \eqref{eq-calculus-valuations}, this does not change \(\sigma_v(\omega)\), for any \(v \in T\) in the direction of \(E_{i_1}\) or \(E_{i_2}\). Therefore, we find that \(\alpha_3, \dots, \alpha_{m} = 0\) and \cref{prop-eps-rup} implies that \(\epsilon_{i_j}(\omega) < 1\), for \(j = 3, \dots, m_i\).

    \jump
        
    Hence, to obtain the lower bounds, assume there exists \(j \in \{3, \dots, m_i\}\) such that \(\epsilon_{i_j}(\omega) \leq -1\). Using \cref{lemma-sumeps},
    \begin{equation} \label{eq-eps-proposition}
        \epsilon_{i_j}(\omega) =  \alpha_{m_i+1} - \epsilon_{i_1}(\omega) - \epsilon_{i_2}(\omega) + m_i - 2 - \sum_{k = 3, k \neq j}^m \epsilon_{i_k}(\omega) \geq - \alpha_{m_i+1} -1
    \end{equation}
    where in the inequality we used the bounds for \(\epsilon_{i_1}(\omega), \epsilon_{i_2}(\omega)\) and the inequalities \(\epsilon_{i_j}(\omega) < 1\), for \(j = 3, \dots, m_i\). Therefore,
    \[
        -\alpha_{m_i+1} - 1 \leq \epsilon_{i_j}(\omega) \leq -1,
    \]
    implies that \(\alpha_{m_i+1} \geq 0\). If \(\alpha_{m_i+1} = 0\), then \(\epsilon_{i_j}(\omega) = -1\) and we are done. Otherwise, we can define a new \(\omega' \in \Omega^2_{\mathbb{C}^2, x}\) by replacing one factor \(f_i\) of \(\omega\) by \(f_{v, j}, v \in T, r_v = 1 \). This way,
    \[ \sigma_i(\omega') = \sigma_i(\omega), \quad \epsilon_{i_j}(\omega') = \epsilon_{i_j}(\omega) + 1 \quad \textnormal{and} \quad \epsilon_{i_k}(\omega') = \epsilon_{i_k}(\omega) \quad \textnormal{for all}\quad k \neq j, \]
    Finally, either \( -1 \leq \epsilon_{i_j}(\omega') < 1 \), or we can replace \( \omega \) by \( \omega' \) and repeat the same process until the lower bound is satisfied. As before, \eqref{eq-calculus-valuations} implies that \(\sigma_v(\omega') = \sigma_v(\omega)\), for all \(v \in T\) appearing in the direction of \(E_{i_1}\) or \(E_{i_2}\).
\end{proof}

\begin{corollary} \label{cor-eps-proposition}
    If \( \epsilon_{i_j}(\omega) = -1 \), for some \(j \in \{3, \dots, m_i\}\), then \(\epsilon_{i_1}(\omega) = \epsilon_{i_2}(\omega) = 1\) and either \( r_i = 2 \) or \( r_i = 3 \) and \( E_i \) is satellite.
\end{corollary}
\begin{proof}
    If \(\epsilon_{i_j}(\omega) = -1\), then necessarily \(\alpha_{m_i+1} = 0\), otherwise we can repeat the same process and define \(\omega'\) such that \(\epsilon_{i_j}(\omega') = \epsilon_{i_j}(\omega) + 1\) and \(\sigma_{i}(\omega') = \sigma_i(\omega)\). After \eqref{eq-eps-proposition}, we see that \(\epsilon_{i_j}(\omega) = -1\) is only possible if \(m = 3\) and \(\epsilon_{i_1}(\omega) = \epsilon_{i_2}(\omega) = 1\). Hence, either \(r_i = 2\) or \(r_i = 3\) and \(E_i\) is satellite.
\end{proof}

When \(r_i \geq 3\), the results from \cref{eps-proposition} are optimal in the sense that it is possible to find examples where \(g\) must necessarily have components \(f_i\), that is \(\alpha_{m_i+1} > 0\) if all bounds are to be satisfied.

\subsection{The exceptional case \texorpdfstring{\((\dag)\)}{(dag)}}

The case of a satellite rupture divisor with \( r_i = 3 \) is somehow especial since we might have \( \epsilon_{i_3}(\omega) = -1 \). We will characterize this case next.

\begin{lemma} \label{lemma-rup3}
    Let \( E_i \) be a satellite divisor with \( r_i = 3 \). Then, with the notations from \cref{lemma-linear-valuations}, \( \beta_1, \beta_2 = 0 \), if and only if
    \begin{equation} \label{eq-lemma-rup3}
         N_i = \lambda n_g \overline{m}_g, \quad N_{i_1} = \lambda a_g \overline{m}_g, \quad \textnormal{and} \quad N_{i_2} = \lambda n_g(c_g n_{g-1} \overline{m}_{g-1} + d_g),
    \end{equation}
    for some \( \lambda \in \mathbb{Z}_{>0} \). 
\end{lemma}
\begin{proof}
    Start assuming that \( \beta_1, \beta_2 = 0 \). Recall that
    \[ N_i = n_1 \beta'_1 + \overline{m}_g \beta_1 = n_1 \beta_2 + \overline{m}_g \beta'_2, \]
    \[ N_{i_1} = a_g \beta'_1 + (a_g n_{g-1} \overline{m}_{g-1} + b_g) \beta_1, \quad \textnormal{and} \quad N_{i_2} = c_g \beta_2 + (c_g n_{g-1} \overline{m}_{g-1} + d_g) \beta'_2. \]
    If \( \beta_1 = 0 \), then \( N_i \) and \( N_{i_1} \) are divisible by \( n_g \) and \( a_g \), respectively. Moreover, \( N_i / N_{i_1} = n_g / a_g \). On the other hand, if \( \beta_2 = 0 \), we have
    \[ N_i = \overline{m}_g \beta'_1, \quad \textnormal{and} \quad N_{i_2} = (c_g n_{g-1} \overline{m}_{g-1} + d_g) \beta'_2. \]
    Therefore, since \( n_g, \overline{m}_g \) are coprime, \( \beta'_1 = \lambda n_g \) for some \( \lambda \in \mathbb{Z}_{>0} \). This proves the first implication. For the other implication, a direct computation using the expressions for \(N_i, N_{i_1}, N_{i_2}\) from \eqref{eq-lemma-rup3} and \(k_i, k_{i_1}, k_{i_2}\) from \cref{lemma-valuations-3} shows that 
    \[
        \epsilon_{i_1}(\dd x) = \frac{1}{n_i}, \quad \textnormal{and} \quad \epsilon_{i_2}(\dd x) = \frac{m_{g-1} + n_1 \cdots n_{g-1} - n_{g-1}\overline{m}_{g-1}}{\overline{m}_g}.
    \]
    Thus, \( \beta_1, \beta_2  = 0 \) by \cref{prop-eps-rup}. 
\end{proof}

Recall that \( C_i \) denoted a plane branch whose strict transform intersects \( E_i \) transversally at a general point. If \( N_i(C_i), N_{i_1}(C_i) \) and \( N_{i_2}(C_i) \) are the corresponding multiplicities of its total transform along \( E_i, E_{i_1} \) and \( E_{i_2} \), then they are equal to \eqref{eq-lemma-rup3} with \(\lambda = 1\) by \cref{lemma-valuations-3}.

\begin{lemma} \label{lemma-rup3-2}
    Let \(E_i\) be a satellite divisor with \(r_i = 3\) and such that
    \[ N_i = \lambda N_i(C_i), \quad N_{i_1} = \lambda N_{i_1}(C_i), \quad \textnormal{and} \quad N_{i_2} = \lambda N_{i_2}(C_i), \]
    for some \( \lambda \in \mathbb{Z}_{>0} \). For any differential form \(\omega \in \Omega^2_{\mathbb{C}^2, x}\), one has that \(\epsilon_{i_1}(\omega), \epsilon_{i_2}(\omega) > 0\). Moreover, if \(\epsilon_{i_3}(\omega) \in \mathbb{Z}\), then \(\epsilon_{i_1}(\omega), \epsilon_{i_2}(\omega) \in \mathbb{Z}\) and \(\epsilon_{i_3}(\omega) \leq -1\).
\end{lemma}
\begin{proof}
    After \cref{lemma-rup3} and \cref{prop-eps-rup},
    \[
        \epsilon_{i_1}(\omega) = \frac{\alpha_1 + 1}{n_g}, \quad \epsilon_{i_2}(\omega) = \frac{\alpha_2 + m_{g-1} + n_1 \cdots n_{g-1} - n_{g-1} \overline{m}_{g-1}}{\overline{m}_g},
    \]
    from which it follows that \(\epsilon_{i_1}(\omega), \epsilon_{i_2}(\omega) > 0\). Now, if \(\epsilon_{i_3}(\omega) \in \mathbb{Z}\), \cref{lemma-sumeps} implies that \(\epsilon_{i_1}(\omega) + \epsilon_{i_2}(\omega) \in \mathbb{Z}\). But since the denominators \(n_g\) and \(\overline{m}_g\) are coprime, this can only happen if \(\epsilon_{i_1}(\omega)\) and \(\epsilon_{i_2}(\omega)\) are also integers. In particular, \(\epsilon_{i_1}(\omega), \epsilon_{i_2}(\omega) \geq 1 \). Finally, \cref{lemma-sumeps} gives 
    \[\epsilon_{i_3}(\omega) \leq 1 - \epsilon_{i_1}(\omega) - \epsilon_{i_2}(\omega),\]
    from which we deduce that \(\epsilon_{i_3}(\omega) \leq -1\).
\end{proof}

\begin{corollary} \label{cor-eps-prop}
    Let \( E_i \) be a satellite divisor with \( r_i = 3 \) and \( \omega \in \Omega^2_{\mathbb{C}^2, x} \) a form obtained from \cref{eps-proposition}. If we have that \( \epsilon_{i_3}(\omega) = -1 \), then there exists \( \omega' \in \Omega^2_{\mathbb{C}^2, x} \) such that \( \epsilon_{i_3}(\omega') = 0 \) and \( \epsilon_{i_j}(\omega') = 0 \), for \( j = 1 \) or \( 2 \), with \( \sigma_i(\omega') = \sigma_i(\omega) \) if and only if
    \begin{equation} \label{eq-cor-eps-prop}
         N_i = \lambda N_i(C_i), \quad N_{i_1} = \lambda N_{i_1}(C_i), \quad \textnormal{and} \quad N_{i_2} = \lambda N_{i_2}(C_i),
    \end{equation}
    does not hold for any \( \lambda \in \mathbb{Z}_{>0} \).
\end{corollary}
\begin{proof}
    Assuming that \eqref{eq-cor-eps-prop} does not hold for any \( \lambda \in \mathbb{Z}_{>0} \), \cref{lemma-rup3} implies that either \( \beta_1 > 0 \) or \( \beta_2 > 0 \). In either case, we can use \cref{prop-rup2,eps-proposition} to construct \( \omega' \in \Omega^2_{\mathbb{C}^2, x} \) with \( \sigma_i(\omega') = \sigma_i(\omega) \) such that \(\epsilon_{i_3}(\omega') = 0\) and either \( \epsilon_{i_1}(\omega') = 0 \) or \( \epsilon_{i_2}(\omega') = 0 \). For the converse implication, if we can find an \(\omega'\) such that \(\epsilon_{i_3}(\omega') = 0\), then it is not possible that \eqref{eq-cor-eps-prop} holds form some \(\lambda \in \mathbb{Z}_{>0}\), since this would contradict \cref{lemma-rup3-2}.
\end{proof}

The numerical condition in \cref{cor-eps-prop} can be understood as the fact that the curve \( C \) does not look like the power of an irreducible curve locally around the exceptional divisor \( E_i \). 


\begin{proof}[Proof of \cref{thm-main-1}]
    The bounds and the last part follow from \cref{prop-rup2,eps-proposition}. The second part follows from \cref{lemma-sumeps} and \cref{cor-eps-proposition,cor-eps-prop}.
\end{proof}

\section{Admissible chains} \label{admissible}

Let \(\omega \in \Omega^2_{\mathbb{C}^2, x}\) be a top differential form. Denote by \(\mathcal{M}(\omega)\) the induced subgraph of the dual graph \(\Delta_\pi\) spanned by those edges \(E_i,E_j\) such that \(\sigma_i(\omega) = \sigma_j(\omega)\). An admissible chain on \(\Delta_\pi\) is a connected path on \(\mathcal{M}(\omega)\) with endpoints being either arrowheads or rupture divisors, and such that
\begin{equation} \label{eq-bound-admissible}
 -1 < \epsilon_{j_k}(\omega) \leq 1, \quad \textnormal{for} \quad k = 1, \dots, m_j,
\end{equation} 
for all \(E_j\) on the path. Moreover, \(\epsilon_{j_k}(\omega) = 1\) can only occur for \(E_{j_1}\), resp. \(E_{j_2}\), if there are no arrowheads in the connected component of \(\Delta_\pi \setminus \{E_j\}\), containing \(E_{j_1}\), resp. \(E_{j_2}\). An admissible path is called maximal if, in addition, both endpoints \(E_j\) of the path satisfy \(\epsilon_{j_k}(\omega) \neq 0 \), for all \(k = 1, \dots, m_j\), except for possibly \(\epsilon_{j_2}(\omega)\).

\begin{proposition} \label{prop-admissible-chains}
If \(\omega \in \Omega^2_{\mathbb{C}^2, x}\) has a residue number vanishing at a rupture divisor \(E_i\), then there is \(\omega' \in \Omega^2_{\mathbb{C}^2, x}\) such that \(\mathcal{M}(\omega')\) has a maximal admissible chain containing \(E_i\).
\end{proposition}

The existence of admissible chains for the standard volume form \(\omega = \dd x\) is a consequence of \cite[Thm. 3]{Vey95}. We will prove \cref{prop-admissible-chains} at the end of the section, but first, we need some intermediate results. 

\jump

Throughout this section, we will keep using the same notations and assumptions as in \cref{section-eps}. The next result is an easy, but key consequence of \cref{lemma-sumeps} and the assumptions made on \(\omega\).

\begin{lemma} \label{lemma-valency2}
    With the assumptions above, let \(E_j, j \in T_e\), be an exceptional divisor. If \(r_j = 2\) and one residue number for \(\omega\) at \(E_j\) vanishes, then so does the other. Moreover, when \(r_j = 1\) the only residue number for \(\omega\) at \(E_j\) is less than or equal to \(-1\).
\end{lemma}

After \cref{lemma-valency2}, it is enough to focus on the case of two exceptional divisors \(E_i, E_j \), not necessarily rupture, connected by a path in \(\mathcal{M}(\omega)\) with no other rupture divisors in between. Assume that \(E_i\) is closer to the blow-up of the origin \(E_0\) than \(E_j\). In this situation, the divisor adjacent to \(E_j\) on the path to \(E_i\) must necessarily be \(E_{j_2}\). 
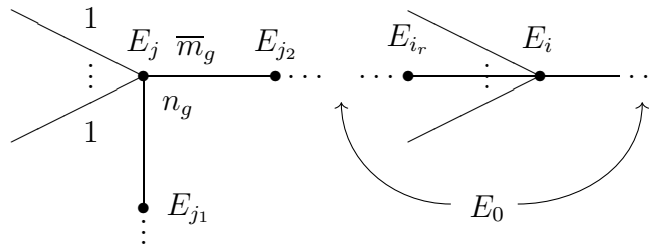
\begin{center}
    \begin{picture}(500,105)(30,-40)
        \put(160,23){\makebox(0,0){\(\vdots\)}}
        \put(180,20){\line(-2,-1){50}}
        \put(180,20){\line(-2,1){50}}
        \put(180,20){\circle*{4}}
        \put(180,20){\line(1,0){50}}
        \put(180,20){\line(0,-1){50}}

        \put(180,-30){\circle*{4}}
        \put(180,-36){\makebox(0,0){\(\vdots\)}}

        \put(230,20){\circle*{4}}
        \put(241,20){\makebox(0,0){\(\dots\)}}

        \put(268,20){\makebox(0,0){\(\dots\)}}
        \put(280,20){\circle*{4}}

        \put(280,20){\line(1,0){50}}
        \put(330,20){\circle*{4}}
        \put(330,20){\line(-2,-1){50}}
        \put(330,20){\line(-2,1){50}}
        \put(310,23){\makebox(0,0){\(\vdots\)}}
        \put(330,20){\line(1,0){30}}
        \put(370,20){\makebox(0,0){\(\dots\)}}

        \put(329,34){\makebox(0,0){\(E_i\)}}
        \put(280,34){\makebox(0,0){\(E_{i_r}\)}}

        \put(180,32){\makebox(0,0){\(E_j\)}}
        \put(230,32){\makebox(0,0){\(E_{j_2}\)}}
        \put(197,-30){\makebox(0,0){\(E_{j_1}\)}}

        \put(193,8){\makebox(0,0){\(n_{g}\)}}
        \put(200,30){\makebox(0,0){\(\overline{m}_{g}\)}}
        \put(160,42){\makebox(0,0){\(1\)}}
        \put(160,-1){\makebox(0,0){\(1\)}}
        
        \put(310, -30){\makebox(0,0){\(E_0\)}}
    \end{picture}
    \begin{tikzpicture}[remember picture, overlay]
        \node (a) at (1.5, 0.9) {};
        \node (b) at (-0.1, 2.4) {};
        \node (d) at (2.1, 0.9) {};
        \node (c) at (3.9, 2.4) {};
        \draw[->] (a)  to [out=-180,in=-90, looseness=1] (b);
        \draw[->] (d)  to [out=0,in=-90, looseness=1] (c);
    \end{tikzpicture}
\end{center}

\begin{proposition} \label{prop-two-rupture}
    With the assumptions above, let \(E_i\) and \(E_j\) be two exceptional divisors connected by a path in \(\mathcal{M}(\omega)\). If there are no other rupture divisors on the path, then there exists \(\omega' \in \Omega^2_{\mathbb{C}^2, x}\) satisfying 
    \begin{equation} \label{eq-prop-two-rupture}
       \sigma_{j}(\omega') = \sigma_{j}(\omega), \qquad \epsilon_{j_2}(\omega') = 0, \qquad \textnormal{and} \qquad -1 < \epsilon_{j_1}(\omega') \leq 1, 
    \end{equation}
    where \( \epsilon_{j_1}(\omega') = 1\) can only occur if there are no arrowheads in the connected component of \(\Delta_\pi \setminus \{E_j\}\) containing \(E_{j_1}\). Moreover, \(\sigma_{v}(\omega') = \sigma_{v}(\omega)\), for all \(v \in T\) in the connected component of \(\Delta_\pi \setminus \{E_j\}\) containing \(E_{j_2}\).
\end{proposition}
\begin{proof}
    Assume that \(E_{i_r}\) is the divisor adjacent to \(E_i\) on the path between \(E_i\) and \(E_j\), that is \(\epsilon_{i_r}(\omega) = 0\). Assuming that \(E_{j_1} \neq 0\), in particular \(E_{j}\) is rupture, we will define a new differential form \(\omega'\) satisfying the conditions \eqref{eq-prop-two-rupture}, otherwise there is nothing to prove. In this setting, we can always find a divisor \(E_k, k \in T,\) in the direction of \(E_{j_1}\) with \(r_k = 1\) and such that any branch \(f_k\) has \(v_j(f_k) = \overline{m}_g\). 
    
    \jump

    We can start assuming that the \(\alpha^{(j)}_3, \dots, \alpha^{(j)}_{m_j+1} \in \mathbb{Z}_{\geq 0}\) from \cref{cor-valuations-3}, applied to \(E_j\) and \(\omega\), are zero. Indeed, any \(f_v \in \mathscr{O}_{\mathbb{C}^2, x}, v \in T,\) contributing to these numbers, that is, with arrowheads in \(E_j\) or in the direction of \(E_{j_3}, \dots, E_{j_r}\), are such that \(v_j(f_v) = n_{g} \overline{m}_{g} \alpha_v\), for some \(\alpha_v \in \mathbb{Z}_{> 0}\). Hence, replacing any such \(f_v\) appearing as component of \(\omega\) by \(f_k^{n_{g+1}\alpha_v}\) we have that \(\sigma_j(\omega)\) remains constant. In addition, after \eqref{eq-calculus-valuations}, \(\sigma_v(\omega)\) also remains constant for all \(v \in T\) in the direction of \(E_{j_2}\). Moreover, we have that \(\alpha_r^{(i)} = \gamma_{k}\), where \(\gamma_k\) denoted the multiplicity of \(f_k\) as factor of \(\omega\).
    
    \jump
    
    Under these assumptions, \cref{prop-eps-rup,lemma-sumeps} imply that
    \begin{equation} \label{eq-prop-two-rupture2}
        \alpha^{(i)}_r + 1 = -\beta_r^{(i)} \sigma_{i}(\omega), \qquad \textnormal{and} \qquad \epsilon_{j_1}(\omega) = -\sum_{k=3}^{m_j} \beta_{k}^{(j)} \sigma_j(\omega) > 0,
    \end{equation}
    where the superscript on \(\beta_r^{(i)}, \beta_k^{(j)} \in \mathbb{Z}_{> 0}\) denotes the corresponding exceptional divisor in \cref{prop-eps-rup}. Indeed, in the case where \(E_{i_r} = E_{i_2}\), the exceptional divisor \(E_0\) must appear in the path between \(E_{j} \) and \(E_{i}\), and since there are no other rupture divisors in this path, the valuations associated with \(E_{j}\) and \(E_{i}\) have only one Puiseux pair, i.e. \(g = 1\).
    
    \jump
    
    Hence, after \eqref{eq-prop-two-rupture2}, it is enough to assume that \(\epsilon_{j_1}(\omega)  > 1\). By definition of \(\beta_r^{(i)}\), see \cref{lemma-valuations-1,lemma-linear-valuations},
    \[
        \beta_r^{(i)} = \frac{1}{\Pi_{i_r}} \sum_{k \in T_s^{(i_r)}} \ell_{ik} N_k,
    \]
    where \(\Pi_{i_r}\) is the product of decorations adjacent to \(E_i\) not in the direction of \(E_{i_r}\), and \(T_s^{(i_r)}\) is the subset of \(T_s\) corresponding to the subgraph rooted at \(E_{i_r}\) after deleting the edge joining \(E_{i_r}\) with \(E_{i}\).  Similarly, denote by \('T_{s}^{(j)}\) the subset of \(T_s\) corresponding to the subgraph rooted at \(E_{j}\) after deleting the edges joining \(E_j\) with \(E_{j_1}\) and \(E_{j_2}\). Define,
    \[
        '\beta_r^{(i)} = \frac{1}{\Pi_{i_r}} \sum_{k \in 'T_s^{(j)}} \ell_{ik} N_k.
    \]
    Then, 
    \begin{equation} \label{eq3-prop-two-rupture}
        \alpha^{(i)}_r + 1 = -\beta_r^{(i)} \sigma_i(\omega) \geq - {'\beta_r^{(i)}} \sigma_i(\omega) = -n_{g} \sum_{s=3}^{m_j} \beta^{(j)}_s \sigma_{j}(\omega) > n_{g},
    \end{equation}
    since \(\sigma_i(\omega) = \sigma_j(\omega)\) and \(\ell_{ik} = n_{g+1} n_g \overline{m}_g \ell_{j_s k}\), for all \(k \in {'T_s^{(j)}}\) and \( s = 3, \dots, m_j\). Therefore, \(\alpha^{(i)}_r \geq n_{g+1}\). We can now define the differential form \(\omega'\) by replacing the factor \(f_{k}^{\gamma_{k}}\) in \(\omega\) by the two factors
    \[
        f_{k}^{\gamma_{k} - n_{g+1}} \qquad \textnormal{and} \qquad f_j,
    \]
    since we had \(\gamma_k = \alpha_{r}^{(i)}\). The same argument above implies that \(\sigma_j(\omega') = \sigma_j(\omega)\) and \(\sigma_v(\omega') = \sigma_v(\omega)\), for all \(v \in T\) in the direction of \(E_{j_2}\). In particular, \(\epsilon_{j_2}(\omega') = 0\).  Finally, \cref{lemma-sumeps} for the divisor \(E_j\) implies that \(\epsilon_{j_1}(\omega') = \epsilon_{j_1}(\omega) - 1\), since \({\alpha'}^{(j)}_{m_j+1} = 1\). Repeating the same process if necessary, we have that \(\omega'\) satisfies
    \[
        \epsilon_{j_2}(\omega') = 0 \quad \textnormal{and} \quad -1 < \epsilon_{j_1}(\omega') \leq 1,
    \]
    as required. We end the proof by noticing that since there are no rupture divisors in the path joining \(E_j\) and \(E_i\), the first inequality in \eqref{eq3-prop-two-rupture} is strict if and only if there is an arrowhead in the connected component of \(\Delta_\pi \setminus \{E_j\}\) containing \(E_{j_1}\). In this case, the argument above still works even if \(\epsilon_{i_1}(\omega) = 1\).
\end{proof}

\begin{proof}[Proof of \cref{prop-admissible-chains}]
Using \cref{thm-main-1}, we can start by replacing \(\omega\) for an \(\omega'\) satisfying the bounds \eqref{eq-bound-admissible} at \(E_i\) and such that \(\sigma_i(\omega') = \sigma_i(\omega)\). Since \(E_i\) is rupture, replacing \(\omega\) for \(\omega'\) does not change the vanishing residue numbers at \(E_i\). Indeed, for the adjacent divisors \(E_{i_3}, \dots, E_{i_{m_i}}\) this follows from \cref{cor-eps-proposition,lemma-rup3-2}. For the other adjacent divisors \(E_{i_1}, E_{i_2}\) the claim is a consequence of \cref{cor2-prop-rup2}. In particular, two differential forms satisfying the bounds \eqref{eq-bound-admissible} for a divisor \(E_i\) have the same vanishing residue numbers at \(E_i\).

\jump

Next, we can always find an exceptional divisor \(E_j, j \in T_e\), not necessarily rupture, and a differential form \(\omega'\) such that \(\epsilon_{j_2}(\omega') \neq 0\) and satisfying \eqref{eq-bound-admissible} at \(E_j\). Indeed, if \(E_i\) is such that \(\epsilon_{i_2}(\omega') = 0\), we can connect \(E_i\) to another divisor \(E_j\) that is either rupture or equal to \(E_0\) by a path of divisors of valency 2 on \(\mathcal{M}(\omega')\), and such that \(\epsilon_{j_2}(\omega') \neq 0\). By replacing \(E_j\) by \(E_i\) and \(\omega'\) by \(\omega\), we can repeat the whole process above and after a finite number of steps \(\epsilon_{j_2}(\omega') \neq 0 \) and \eqref{eq-bound-admissible} are satisfied at \(E_j\).

\jump

Therefore, assume that \(E_j, j \in T_e,\) is an exceptional divisor on \(\mathcal{M}(\omega')\), not necessarily rupture, such that \(\epsilon_{j_2}(\omega') \neq 0 \) and \(\omega'\) satisfies the bounds \eqref{eq-bound-admissible} at \(E_j\). Now, we can start constructing the maximal admissible path from \(E_j\) since there must exist at least one \(r \in \{3, \dots, m_j\}\) with \(\epsilon_{j_r}(\omega') = 0\). Connect \(E_j\) to some divisor \(E_k, k \in T,\) in the direction \(E_{j_r}\) by a path of divisors of valency 2 on \(\mathcal{M}(\omega')\). If \(E_k\) is rupture, we can use \cref{prop-two-rupture,eps-proposition} to get an \(\omega''\) satisfy \eqref{eq-bound-admissible} for both \(E_k\) and \(E_j\). Replacing \(E_k\) by \(E_j\) and \(\omega'\) by \(\omega''\) we can keep enlarging this path until it is maximal. 

\jump

If the initial divisor \(E_j\) has more than one adjacent edge in \(\mathcal{M}(\omega')\), then two admissible chains must be constructed starting from \(E_j\) to obtain a maximal chain. The conclusions of \cref{prop-two-rupture,eps-proposition} ensure that the differential forms for both branches of the maximal admissible chain starting at \(E_j\) path can be constructed independently.

\jump

We end this proof by noticing that the maximal admissible path we constructed can always be chosen to contain the rupture divisor \(E_i\). If at the start \(E_i\) was changed for another divisor \(E_j\) closer to \(E_0\), this was done following a path of exceptional divisors with several differential forms satisfying the bounds \eqref{eq-bound-admissible} independently. Remembering the initial path from \(E_i\) to \(E_j\), we can trace back the same path from \(E_j\) to \(E_i\) using now a single differential form satisfying \eqref{eq-bound-admissible} for all vertices in that path.

\end{proof}

\section{Periods of integrals in the Milnor fiber} \label{expansions}

For this section, we work in arbitrary dimension \( n +1 \). Let \( f : (\mathbb{C}^{n+1}, x) \longrightarrow (\mathbb{C}, 0) \) be a germ of a holomorphic function. For \( 0 < \delta \ll \epsilon \ll 1 \), let \( \overline{B}_{\varepsilon}(x) \subset \mathbb{C}^n \) be the closed ball of radius \( \varepsilon > 0 \) centered at \( x \). Denote by \( D \subset\mathbb{C} \) the disk of radius \( \delta > 0 \) centered at \( 0 \) and \( D' \) the punctured disk. Set
\[ X = \overline{B}_{\varepsilon}(x) \cap f^{-1}(D), \qquad X' = X \setminus f^{-1}(0), \qquad X_t = \overline{B}_{\varepsilon}(x) \cap f^{-1}(t), \quad t \in D. \]
The restriction \( f': X' \rightarrow D' \) is a locally trivial differentiable fibration such that the diffeomorphism type of the fiber \( X_t, t \in D', \) is independent of the choice of \( \delta, \epsilon \) \cite{Mil68}. The general fiber \( X_t, t \in D', \) is called the Milnor fiber of \( f \). Denote by \( \partial X_t \) the Milnor fiber boundary. The homology groups \(H_i(X_t, \mathbb{C})\) are finite-dimensional vectors spaces vanishing above degree \(n\) since \(X_t\) has the homotopy type of a finite CW-complex of dimension \(n\), \cite[Thm. 5.1]{Mil68}.

\jump

A vanishing cycle is an element \( \gamma(t) \) of \( H_{n}(X_t, \mathbb{C}) \) such that \( \gamma(t) \) collapses to \( x \in X_0 \) when \( t \) tends to zero. Similarly, a relative vanishing cycle \( \gamma(t) \) is defined to be a locally finite cycle in \( H_{n}^{lf}(X_t \setminus \partial X_t, \mathbb{C}) \) such that the limiting cycle \( \gamma(0) \subset X_0 \) goes through \( x \in X_0 \). Recall that locally finite homology, also known as Borel-Moore homology, is defined as homology with closed supports. In contrast, the usual singular homology agrees with homology with compact supports.

\jump

The fundamental group \( \pi_1(D', 0) \) of the base induces an endomorphism \( h : X_t \longrightarrow X_t \) known as the geometric monodromy. By functoriality, this map induces a morphism on the homology groups \( h_*: H_*(X_t, \mathbb{C}) \longrightarrow H_*(X_t, \mathbb{C}) \), the algebraic monodromy. In the case of locally finite homology, the isomorphism with relative homology \( H_*^{lf}(X_t \setminus \partial X_t, \mathbb{C}) \cong H_*((X_t, \partial X_t), \mathbb{C}) \) and the fact that the geometric monodromy can be constructed to preserve \( \partial X_t \), gives an induced monodromy endomorphism on \( H^{lf}_*(X_t \setminus \partial X_t, \mathbb{C}) \) that we keep denoting by \( h_* \).

\jump

For top dimensional cycles on \( X_t \), there is an inclusion \( H_n(X_t, \mathbb{C}) \hookrightarrow H_n((X_t, \partial X_t), \mathbb{C}) \), so every vanishing cycle is in particular a relative vanishing cycle. Let \( \eta \in \Gamma(X, \Omega^{n}_{X}) \) be a holomorphic \( n \)-form, then
\[ I(t) = \int_{\gamma(t)} \eta, \qquad t \in D', \]
is well-defined since the restriction of \( \eta \) to \( X_t \) is a form of maximal degree. For relative cycles not coming from cycles in \( H_n(X_t, \mathbb{C}) \), the integrals \( I(t) \) must be understood as path integrals starting and ending in the boundary \( \partial X_t \).

\jump

When \( \gamma(t) \) is a vanishing cycle the integrals \( I(t) \) are well-studied in the literature; they are multivalued functions on \( D' \) satisfying asymptotic expansions that are related to the local Bernstein-Sato polynomial \( b_{f, x}(s) \), see for instance \cite{Mal74a,Mal74b,Var80,Var82a}. For relative vanishing cycles similar results hold as we will discuss next.

\jump

Leray's residue theorem states that,
\[ \int_{\gamma(t)} \eta = \frac{1}{2 \pi \imath} \int_{\delta \gamma(t)}{\frac{\dd f \wedge \eta}{f-t}}, \]
where \( \delta : H_{n}(X_t, \mathbb{C}) \longrightarrow H_{n+1}(X \setminus X_t, \mathbb{C}) \) is the Leray coboundary operator, see \cite[III.2.4]{Pham11} for a proof. Since Leray's coboundary operator can be constructed relative to any family of supports, see \cite[II.2.2]{Pham11}, the proof of the residue theorem in \cite[III.2.4]{Pham11} generalizes to locally finite homology. We can deduce that \( I(t) \) is holomorphic in sectors since the integrand depends holomorphically on \( t \) and the derivative is given by
\begin{equation} \label{derivative-integral}
    I'(t) = \frac{d}{d t} \int_{\gamma(t)} \eta = \int_{\gamma(t)} \frac{\dd \eta}{\dd f},
\end{equation}
where \( \dd \eta / \dd f \) is the Gelfand-Leray form. From an elementary point of view \(\dd \eta / \dd f\) is a family of forms on \(X_t, t \in D',\) that are defined as follows: on the set \(\{\frac{\partial f}{\partial x_i} \neq 0\}\), \(\dd \eta / \dd f\) is the restriction to \(X_t\) of the form
\[
  (-1)^{i-1} \frac{\partial f}{\partial x_i}^{-1} g \dd x_0 \wedge \cdots \wedge \widehat{\dd x_i} \wedge \cdots \wedge \dd x_{n},  
\]
where \(\eta = g \dd x\). The identity \eqref{derivative-integral} follows from Leray's residue theorem and differentiation under the integral sign,
\begin{equation}
    \begin{split}
        \frac{d}{dt} \int_{\gamma(t)} \eta & = \frac{d}{dt} \left( \frac{1}{2 \pi \imath} \int_{\delta \gamma(t)} \frac{\dd f \wedge \eta}{f - t} \right) = \frac{1}{2 \pi \imath} \int_{\delta \gamma(t)} \frac{\dd f \wedge \eta}{(f - t)^2} \\
        & = \frac{1}{2 \pi \imath} \int_{\delta \gamma(t)} \left\{ \frac{\dd \eta}{f - t} - \dd \left( \frac{\eta}{f - t} \right) \right\} = \frac{1}{2 \pi \imath} \int_{\delta \gamma(t)} \frac{\dd \eta}{f - t} = \int_{\gamma(t)} \frac{\dd \eta}{\dd f},
    \end{split}
\end{equation}
since if \( t \) moves sufficiently close to a neighborhood of a point the representative of the cycle \( \delta \gamma(t) \) can be chosen independent of \( t \), see the proof of \cite[Prop. X.1.1]{Pham11}.

\jump

Assume now that \( (h_* - \lambda)^p \gamma(t) = 0 \). The presence of monodromy implies that the integrals \( I(t) \) have the following expansion
\[ I(t) = \int_{\gamma(t)} \eta = \sum_{\alpha, k} a_{\alpha, k}(\eta)\, t^{\alpha} \frac{(\ln{t})^k}{k!}, \qquad t \in D', \]
where \( \alpha \in \mathbb{C} \) is such that \( e^{-2 \pi \imath \alpha} = \lambda \), and \( 0 \leq k < q \). If \( \gamma(t) \) is a vanishing cycle it follows from the Monodromy Theorem that \( \alpha \in \mathbb{Q} \) and \( q \leq n \). Moreover, from the regularity of the Gauss-Manin connection and a result of Malgrange \cite[Prop. 3.1]{Mal74b} one has that \( \alpha > 0 \) for any holomorphic form \( \eta \in \Gamma(X, \Omega^{n}_X) \). For relative vanishing cycles the same results still hold and can be deduced from the arguments in \cref{prop-bfct} below and the fact that \( b_{f, x}(s) \) has roots in \( \mathbb{Q}_{<0} \) with multiplicity at most \( n + 1 \), \cite{Kas76,Sai94}.

\jump

Since top differential forms \( \omega \in \Gamma(X, \Omega^{n+1}_{X}) \) are exact, in the sequel we will work with period integrals of the form
\begin{equation} \label{eq-topforms}
    \int_{\gamma(t)} \frac{\omega}{\dd f} = \sum_{\alpha, k} a_{\alpha, k}(\omega)\, t^{\alpha-1} \frac{(\ln{t})^k}{k!}, \qquad t \in D'.
\end{equation}

\jump

The next proposition is based on the integration by parts argument from \cite{Mal74b}.

\begin{proposition} \label{prop-bfct}
    Let \( \gamma(t) \in H^{lf}_n(X_t \setminus \partial X_t, \mathbb{C}) \) and let \( \alpha \in \mathbb{C} \) be one of the exponents appearing in \eqref{eq-topforms}. Define,
    \[
        \alpha_0 = \min \{\alpha' \in \alpha + \mathbb{Z} \ |\ \exists\, \omega \in \Gamma(X, \Omega^{n+1}_X) \ \textnormal{with}\ a_{\alpha', k}(\omega) \neq 0 \ \textnormal{for some}\ k \in \mathbb{Z}_{\geq 0} \}.
    \]
    For any \( \omega_0 \in \Gamma(X, \Omega^{n+1}_X) \) attaining the minimum, let \( k_0 = \max \{k \in \mathbb{Z}_{\geq 0} \ |\  a_{\alpha_0, k}(\omega_0) \neq 0\} \). Then, \( -\alpha_0 \) is a root of \( b_{f, x}(s) \) with multiplicity at least \( k_0 + 1 \).
\end{proposition}
\begin{proof}
    After a rescaling, we can always assume that \( 1 \in D \). Take a trivialization of \( f : X \cap f^{-1}((0, 1]) \longrightarrow (0, 1] \) so if \( \gamma(t) \in H^{lf}_n(X_t \setminus \partial X_t, \mathbb{C}) \) is represented by a differentiable chain \( \Gamma(t) \) in \( X_t \), locally, we can assume \( \Gamma(t) = \{t\} \times \Gamma_0 \), for some \( n \)-dimensional differentiable chain \( \Gamma_0 \). Define \( \Delta(t) = [t, 1] \times \Gamma_0 \). Then, for \( t_0 \in (0, 1] \),
    \[ \int^{1}_{t_0} t^s dt \int_{\Gamma(t)} \frac{\omega_0}{\dd f} = \int_{\Delta(t_0)} f^s \omega_0. \]
    Using the defining functional equation of \( b_{f, x}(s) \),
    \[ b_{f, x}(s) \int^{1}_{t_0} t^s dt \int_{\Gamma(t)} \frac{\omega_0}{\dd f} = \int_{\Delta(t_0)} (P(s) f^{s+1}) \omega_0. \]
    We have the identity \( (P(s) f^{s+1}) \omega_0 = f^{s+1} (P^*(s) \omega_0) + \dd (f^{s+1} \eta_0) \),where \( P^*(s) \) is the adjoint operator and \( \eta_0 \in \Omega^n_X[s] \). By Stoke's theorem,
    \[ \int_{\Delta(t_0)} (P(s) f^{s + 1}) \omega_0 = \int_{\Delta(t_0)} f^{s+1} P^*(s) \omega_0 + \int_{\partial \Delta(t_0)} f^s \eta_0. \]
    The boundary of \( \Delta(t_0) \) is \( -\Gamma(t_0), \Gamma(1), -[t_0, 1] \times \partial \Gamma_0 \), the last piece being zero if \( \gamma(t) \in H_n(X_t, \mathbb{C}) \).  Finally,
    \[ b_{f, x}(s) \int^{1}_{t_0} t^s \dd t \int_{\Gamma(t)} \frac{\omega_0}{\dd f} = \int^{1}_{t_0} t^{s+1} \int_{\Gamma(t)} \frac{P^*(s) \omega_0}{\dd f} + \int_{\Gamma(1)} \eta_0 - t_0^{s+1} \int_{\Gamma(t_0)} \eta_0 - \int^{1}_{t_0} t^{s+1} dt \int_{\partial \Gamma(t)} \frac{\eta_0}{\dd f}. \]
    By assuming that \( \textnormal{Re}(s) \gg 0 \) we can take the limit \( t_0 \rightarrow 0 \). Using Stoke's theorem and the fact that \( \dd(\eta_0 / \dd f) = \dd \eta_0 / \dd f \),
    \begin{equation} \label{eq-prop-bfct}
        b_{f, x}(s) \int^{1}_{0} t^s \dd t \int_{\Gamma(t)} \frac{\omega_0}{\dd f} = \int^{1}_{0} t^{s+1} \int_{\Gamma(t)} \frac{P^*(s) \omega_0}{\dd f} + \int_{\Gamma(1)} \eta_0 - \int^{1}_{0} t^{s+1} dt \int_{\Gamma(t)} \frac{\dd \eta_0}{\dd f}.
    \end{equation}
    The asymptotic expansions of the integrals in \eqref{eq-prop-bfct} and the identity
    \[ \int_{0}^1 t^{s + \alpha - 1} (\ln{t})^q dt = \frac{d^q}{ds^q} \frac{1}{(s + \alpha)}, \]
    gives that the left-hand side, divided by \( b_{f, x}(s) \), has a pole of order \( k_0 + 1 \) at \( -\alpha_0 \). The poles from the right-hand side are either not congruent with \( \alpha_0 \) modulo \( \mathbb{Z} \) or if they are, they are smaller or equal to \( -\alpha_0 - 1 \). We conclude that \( - \alpha_0 \) is a root of \( b_{f, x}(s) \) with multiplicity at least \( k_0 + 1 \).
\end{proof}

\begin{remark}
    If \( \gamma(t) \) is a vanishing cycle it follows from the arguments in \cite{Mal74b} that we can replace \( b_{f, x}(s) \) by the reduced local Bernstein-Sato polynomial \( \widetilde{b}_{f, x}(s) \) in the statement of \cref{prop-bfct}.
\end{remark}

\section{Proof of Theorem B}

To show that a given rational number is a root of the local Bernstein-Sato polynomial we will construct the asymptotic expansion from \cref{expansions} using resolution of singularities. In this section, we will continue to work in arbitrary dimension \( n+1 \) unless stated otherwise. We extend the terminology about log resolutions from \cref{intro} to arbitrary dimension.

\subsection{Semi-stable reduction}

Let \( e = \lcm \{N_i \ |\ i \in T\} \) and consider the map \( \widetilde{D} \rightarrow D \) given by \( \tilde{t} \mapsto \tilde{t}^e = t \). Set \( \overline{X} = \pi^{-1}(X) \) and consider \( \widetilde{X} \) the normalization of the fiber product \( \overline{X} \times_D \widetilde{D} \). The space \( \widetilde{X} \) has only orbifold singularities concentrated in codimension at least 2, see \cite{Steen76}. The projection \( \tilde{f} : \widetilde{X} \longrightarrow \widetilde{D} \) has a reduced special fiber. The projection \( \rho : \widetilde{X} \longrightarrow \overline{X} \) is a cyclic cover of degree \(e\) ramified along \( F_\pi \). For \( p \in F_\pi \), \( | \rho^{-1}(p)| = \gcd (N_i \ |\ p \in E_i) \). Define \( \widetilde{F}_\pi = \rho^*(F_\pi) \) and \( \widetilde{E}_i = \rho^*(E_i) \) which are reduced divisors with orbifold normal crossings.

\begin{equation*}
    \begin{tikzcd}
        \widetilde{X} \arrow{d}{\tilde{f}} \arrow{r}{\rho} & \overbar{X} \arrow{d}{\pi^* f} \arrow{r}{\pi} & X \arrow{d}{f} \\
        \widetilde{T} \arrow{r} & T \arrow[r, equal] & T.
    \end{tikzcd}
\end{equation*}

\subsection{Asymptotic expansions terms} \label{subsection-mult1}

Let us fix an exceptional divisor \( E_i, i \in T_e, \) and define \( E_i^\circ = E_i \setminus \cup_{j \neq i} (E_i \cap E_j) \), and similarly for \(\widetilde{E}_i\) in \(\widetilde{X}\). Let \( \omega = g \dd x \) be a top differential form on \( X \) and set \( \widetilde{\omega} = (\pi \rho)^*\omega \). Consider \( \widetilde{X}_i \subset \widetilde{X} \) an open tubular neighborhood of the divisor \( \widetilde{E}_i \) and define \( \widetilde{X}_i^\circ = \widetilde{X}_i \cap (\widetilde{X} \setminus \widetilde{F}_\pi) \). Locally on one of the connected components of \(\widetilde{X}_i^\circ\), assume that the exceptional divisor \( E_i \) has equation \( y_0 \). Then, we can decompose \( \widetilde{\omega} \) as
\begin{equation} \label{decomposition1-eq}
    \widetilde{\omega} = \widetilde{\omega}_0 + \widetilde{\omega}_1 + \cdots + \widetilde{\omega}_{\nu} + \cdots, \qquad \nu \in \mathbb{Z}_{\geq 0},
\end{equation}
by considering a series expansion with respect to \(y_0\), and where \( \widetilde{\omega}_\nu \) is a local section of \( \Omega_{\widetilde{X}}^{n+1}(-\nu \widetilde{E}_i) \). A computation in local coordinates shows that the multiplicity of \(\widetilde{\omega}_\nu\) along \(\widetilde{E}_i\) equals
\[ v_i(\widetilde{\omega}_\nu) = e\frac{k_i + v_i(g) + \nu}{N_i} - 1. \]
Hence, \( y_0^{-v_i(\widetilde{\omega}_\nu)} \widetilde{\omega}_\nu \) is defined on \( \widetilde{X}^\circ_i \) and extends holomorphically to \( \widetilde{E}_i \). For our purposes, we will only be interested in the first term \( \widetilde{\omega}_0 \), which is always non-zero. Locally on \( \widetilde{X}^\circ_i \) we can write \( \tilde{f}(y_0, \dots, y_n) = y_0 = \tilde{t} \), hence \( \dd \widetilde{f} = \dd y_0 \), and
\begin{equation} \label{pullback-eq}
    \frac{\widetilde{\omega}}{\dd \tilde{f}} = \sum_{\nu \geq 0} \tilde{t}^{\,v_i(\widetilde{\omega}_\nu)} R_i(\widetilde{\omega}_\nu).
\end{equation}
Here, \( R_i(\widetilde{\omega}_\nu) = \tilde{f}^{-v_i(\widetilde{\omega}_\nu)} \widetilde{\omega}_\nu / \dd \tilde{f} \) is a holomorphic form on \( \widetilde{X}_i^\circ \) extending holomorphically to \( \widetilde{E}_i^\circ \) and that is locally constant with respect to \( \tilde{t} = y_0 \). Using the relation,
\begin{equation} \label{easy-eq}
    \restr{\frac{\rho^* \overline{\omega}}{\dd \tilde{f}}}{\widetilde{X}_{\tilde{t}}} = e \tilde{f}^{e-1} \rho^* \left(\restr{\frac{\overline{\omega}}{\dd \overbar{f}}}{\overline{X}_{\tilde{t}^e}}\right),
\end{equation}
we can push down to \( \overline{X} \) and obtain
\begin{equation}
    \frac{\overline{\omega}}{\dd \bar{f}} = \sum_{\nu \geq 0} t^{-\sigma_{i, \nu}(\omega) -1} R_{i, \nu}(\omega), \quad \textnormal{where} \quad \sigma _{i, \nu}(\omega) = -\frac{k_i + v_i(g) + \nu}{N_i},
\end{equation}
locally on \( \overline{X}_i^\circ \), where \( R_{i, \nu}(\omega) = e^{-1} (\rho^{-1})^* {R_i(\widetilde{\omega}_\nu)} \). Notice that for \( \nu = 0 \) we recover the candidate roots \( \sigma_i(\omega) \) from \cref{intro}. In the sequel, we will drop the dependence on the index \( \nu \) whenever \( \nu = 0 \). The restriction of \(R_{i, \nu}(\omega)\) to \(E_i^\circ\) defines a multivalued \( n \)-form and, focusing on the case \( \nu = 0 \), the order of vanishing of the restriction of \(R_{i}(\omega) \) to \(E_{i}^\circ\) along \( E_{\{i,j\}} = E_i \cap E_{j} \neq \emptyset \) equals \( \epsilon_{i_j}(\omega) - 1 \), see \cite[Lemma 10.5]{Bla21} or \cite[\S III.3]{Loe88}.

\subsection{Multivalued forms in \texorpdfstring{\(\mathbb{P}^1_{\mathbb{C}} \setminus S\)}{P1C \ S}} \label{multivalued}

For this section we will set \(n = 1\), and construct a vanishing cycle \( \gamma(t) \in H_1(X_t, \mathbb{C}) \) such that
\[ \lim_{t \rightarrow 0}\, t^{\sigma_{i}(\omega) + 1} \int_{\gamma(t)} \frac{\omega}{\dd f} = a_{-\sigma_i(\omega), 0} \neq 0. \]
In this situation, \( E_i \cong \mathbb{P}^1_{\mathbb{C}} \) and \( R_i(\omega) \) defines a multivalued form on \( E_i^{\circ} \cong \mathbb{P}^1_{\mathbb{C}} \setminus S \), where \( S = \{p_1, \dots, p_{r_i}\} \). We will use the same argument as in \cite{Loe88}. Let \( L \) be the local system on \( \mathbb{P}^1_{\mathbb{C}} \setminus S \) with monodromies \( e^{- 2 \pi \imath \epsilon_{i_j}(\omega)} \) around \( p_j \in S \). Let \( p: F \longrightarrow E_i^\circ \) be the finite cover associated with \( L \) and denote by \( L^\vee \) the dual local system to \( L \). For ease of notation denote \( \mu_{i_j} = 1- \epsilon_{i_j}(\omega) \).

\begin{proposition}[{\cite[Prop. 2.14]{DM86}}] \label{deligne-mostow}
    Assume that \( r_i \geq 3 \), \( \epsilon_{i_j}(\omega) \not\in \mathbb{Z} \) for \( j = 1, \dots r \), and \( \sum_{j = 1}^r \mu_{i_j} = 2 \). Then, every section of \( \Gamma(\mathbb{P}^1_{\mathbb{C}}, \Omega^1(\sum_{j = 1}^r \mu_{i_j} p_j)(L)) \) defines a non-zero cohomology class in \( H^1(\mathbb{P} \setminus S, L) \).
\end{proposition}

Since the restriction of \( R_i(\omega) \) to \( E_i^\circ \) defines a section of \( \Gamma(\mathbb{P}^1_{\mathbb{C}}, \Omega^1(\sum_{j = 1}^r \mu_{i_j} p_j)(L))) \), and assuming that the hypothesis of \cref{deligne-mostow} hold, we have the existence of a twisted cycle \( C \in H_1(E_i^\circ, L^\vee) \) such that
\[ \int_{C} R_i(\omega) \neq 0. \]
By definition, the twisted cycle \( C \in H_1(E_i^\circ, L^\vee) \) gets identified with a cycle in \( H_1(F, \mathbb{C}) \). The restriction of the morphism \( \rho \) to \( \widetilde{E}_i \) is a ramified covering of degree \( N_i \) ramified at the intersection points \( E_i \cap E_j \) with monodromies \( \exp({2 \pi \imath {N_j}/{N_i}}) \). Since the monodromies of \( F \) are
\[ \exp{(2 \pi \imath \epsilon_{i_j}(\omega))} = \exp{\left(-2 \pi \imath (k_i + v_i(g)) \frac{N_j}{N_i}\right)} = \exp{\left(2 \pi \imath \frac{N_j}{N_i} \right)}^{-(k_i + v_i(g))}, \]
the restriction of \( \rho \) to \( \widetilde{E}^\circ_i \) factorizes through a connected component \( F_0 \) of \( F \), namely
\[
    \restr{\rho}{\widetilde{E}^\circ_i} : \widetilde{E}_i^\circ \xrightarrow{\quad q \quad } F_0 \xrightarrow{\ \ {{p|}_{F_0}}\ \ } E_i^\circ.
\]
Being \( q \) also a finite cover, there exists \( m \in \mathbb{Z} \setminus \{0\} \) and a cycle \( \tilde{\gamma} \in H_1(\widetilde{E}_i^\circ, \mathbb{C}) \) such that \( q_* \tilde{\gamma} = m C \). Finally, we can use the local triviality of \( \tilde{f} \) in a neighborhood of \( \widetilde{E}^\circ_i \) to extend \( \tilde{\gamma} \) to a family of locally constant cycles \( \tilde{\gamma}(\tilde{t}) \) in \( H_1(\widetilde{X}_{\tilde{t}}, \mathbb{C}) \), for \( \tilde{t} \in \widetilde{D}' \), such that \( \tilde{\gamma}(0) = \tilde{\gamma} \).

\jump

Setting \( \gamma(\tilde{t}^e) = \rho_* \tilde{\gamma}(\tilde{t}) \), we have constructed the desired vanishing cycle and the above argument can be summarized as
\begin{align}
    \nonumber \int_{\gamma(t)} \frac{\overline{\omega}}{\dd \bar{f}} = \int_{\rho_* \tilde{\gamma}(\tilde{t})} \restr{\frac{\overline{\omega}}{\dd \bar{f}}}{\overline{X}_{\tilde{t}^e}} & = \int_{\tilde{\gamma}(\tilde{t})} \rho^* \left(\restr{\frac{\overline{\omega}}{\dd \bar{f}}}{\overline{X}_{\tilde{t}^e}}\right)                                                    \\
    \nonumber                                                                                                                                                                            & = e^{-1}\tilde{t}^{\,1-e} \int_{\tilde \gamma(\tilde{t})} \restr{\frac{\rho^* \overline{\omega}}{\dd \tilde{f}}}{\widetilde{X}_{\tilde{t}}}       &  & \eqref{easy-eq}              \\
                                                                                                                                                                                         & = \tilde{t}^{\,e(-\sigma_i(\omega) - 1)} \int_{\tilde{\gamma}(\tilde{t})} e^{-1} R_i(\widetilde{\omega}_0) + \cdots                                &  & \eqref{pullback-eq}          \\
                                                                                                                                                                                         & = \tilde{t}^{\,e(-\sigma_i(\omega) - 1)} \int_{\tilde{\gamma}(\tilde{0})} e^{-1} {R_i(\widetilde{\omega}_0)} + \cdots &  & \textnormal{(loc. constant)} \\
    \nonumber                                                                                                                                                                            & = \tilde{t}^{\,e(-\sigma_i(\omega) - 1)} \int_{\rho_* \tilde{\gamma}} R_i(\omega) + \cdots                                                                                           \\
    \nonumber                                                                                                                                                                            & = t^{-\sigma_i(\omega) -1} \, m \int_{C} R_i(\omega) + \cdots                                                                                      &  & (q_* \tilde{\gamma} = m C)
\end{align}
where the dots indicate higher-order terms in \( t \).

\subsection{Logarithmic terms. Part I} \label{subsection-mult2}
For this section, we keep working in dimension \( n + 1\). Let \( I \subseteq T, |I| = r + 1, \) such that \( E_I = \cap_{i \in I} E_i \neq \emptyset \) and that \( \sigma_{i}(\omega) = \sigma_{j}(\omega) \), for all \( i, j \in I \). Let us call this number \(\sigma_I (\omega)\). Notice that this is equivalent to \( \epsilon_{i_j}(\omega) = 0 \), for all \( i, j \in I \). If \(d_I = \gcd(N_i \ | \ i \in I)\), then from Bezout's theorem it follows that \( \sigma_I(\omega) \in \mathbb{Z}[\frac{1}{d_I}] \). Define \( E_I^\circ = E_I \setminus \cup_{j \not\in I} (E_I \cap E_j) \) and similarly for \(\widetilde{E}_I^\circ\) in \(\widetilde{X}\). 

\jump

If we keep denoting \(\widetilde{\omega} = (\rho \pi)^* \omega\), this setting allows a decomposition similar to \eqref{decomposition1-eq}. Let \( \widetilde{X}_I \) be an open tubular neighborhood of \( \widetilde{E}_I \) and similarly define \( \widetilde{X}_I^\circ = \widetilde{X}_I \cap (\widetilde{X} \setminus \widetilde{F}_\pi) \). Working locally on one connected component of \(\widetilde{E}^\circ_I\) and assuming that the variety \( \widetilde{E}_I \) is defined by the equations \( \{ y_i = 0 \}_{i \in I} \), we can decompose \( \widetilde{\omega} \) as
\[ \widetilde{\omega} = \sum_{\underline{\nu} \in S_I(\omega)} \widetilde{\omega}_{\underline{\nu}},\]
by consider the series expansion with respect to the variables \( \{y_i\}_{i \in I} \), and where
\[S_I(\omega) = \left\{ (\nu_i)_{i \in I} \in \mathbb{Z}_{\geq 0}^{r+1} \ \bigg|\ \frac{k_i + v_i(g) + \nu_i}{N_i} = \frac{k_j + v_j(g) + \nu_j}{N_j},\quad \forall i, j \in I \right\}. \]
In this case, each \( \widetilde{\omega}_{\underline{\nu}} \) is a local section of \( \Omega_{\widetilde{X}}^{n+1}(-\sum_{i \in I} \nu_i \widetilde{E}_i) \) with multiplicity along \(\widetilde{E}_i\) equal to
\[ v_i(\widetilde{\omega}_{\underline{\nu}}) = e\frac{k_i + v_i(g) + \nu_i}{N_i} - 1, \quad \textnormal{for all} \quad i \in I. \]
Similarly, \( \prod_{i \in I} y_i^{-v_i(\widetilde{\omega}_{\underline{\nu}_k})} \widetilde{\omega}_{\underline{\nu}_k} \) is defined on \( \widetilde{X}^\circ_I \) and extends holomorphically to \( \sum_{i \in I} \widetilde{E}_i \cap \widetilde{X}_I \). Again, we will only be interested in the first term \( \omega_{\underline{0}} \) which is always non-zero. Locally on \( \widetilde{X}_I^\circ\) one has \( \tilde{f}(y_0, \dots, y_n) = \prod_{i \in I} y_i = \tilde{t} \), hence
\[ \dd \tilde{f} = \sum_{i \in I} \prod_{j \in I \setminus \{i\}} y_j \dd y_i. \]
Choosing \( i \in I \), a computation in local coordinates shows that
\begin{equation} \label{pullback-eq2}
    \frac{\widetilde{\omega}_{\underline{\nu}}}{\dd \tilde{f}} = \sum_{\underline{\nu} \in S_I(\omega)} \tilde{t}^{\,v_I(\widetilde{\omega}_{\underline{\nu}})} R_i(\widetilde{\omega}_{\underline{\nu}}), \quad \textnormal{where} \quad v_I(\widetilde{\omega}_{\underline{\nu}}) = e\frac{k_i + v_i(g) + \nu_i}{N_i} - 1,
\end{equation}
where \( R_i(\widetilde{\omega}_{\underline{\nu}}) = \tilde{f}^{-v_I(\tilde{\omega}_{\underline{\nu}})} \widetilde{\omega}_{\underline{\nu}} / \dd \tilde{f} \) is a holomorphic form on \( \widetilde{X}^\circ_I \)
that extends holomorphically to \( \widetilde{E}_i^\circ \cap \widetilde{X}_I \). Moreover, it is locally constant with respect to \( y_i \). However, the restriction of \(R_i(\widetilde{\omega}_{\underline{\nu}})\) to the fiber \( \prod_{j \in I} y_j = \tilde{t}\) does not depend on the choice of \(i \in I\).

\jump

In the sequel, we will focus on the case \( \underline{\nu} = 0 \). We can assume without loss of generality that \( E_I \) has one connected component. Since \( \epsilon_{i_j}(\omega) = 0 \), the restriction of \( R_i(\widetilde{\omega}_{\underline{0}}) \) to \(\widetilde{E}_i^\circ \cap \widetilde{X}_I\) has logarithmic poles along \( \sum_{j \in I \setminus \{i\}} \widetilde{E}_j \cap \widetilde{E}_i \), see \cite[Lemma 4.7]{Var82a}. Even in the presence of orbifold singularities we can apply the residue map, see \cite[\S 1]{Steen76}, to obtain a non-zero form \(\textnormal{Res}_{\widetilde{E}_I}[R_i(\widetilde{\omega}_{\underline{0}})] \) on \( \widetilde{E}^{\circ}_I \), that does not depend on the choice of \( i \in I\). As before, defining \[ \textnormal{Res}_{E_I}(\omega) = e^{-1} \big(\rho^{-1})^* \,\textnormal{Res}_{\widetilde{E}_I}[R_i(\widetilde{\omega}_0)], \]
we obtain a non-zero multivalued differential form on \( E_I^\circ \) that as order of vanishing equal to \( \epsilon_{I_j}(\omega) - 1 = \epsilon_{i_j}(\omega) - 1 \) along \( E_I \cap E_j \neq \emptyset, j \not\in I \).

\jump

Generalizing the construction from \cref{multivalued}, let \( L \) be the local system on \( E_I^\circ \) with monodromies \( e^{-2 \pi \imath \epsilon_{I_j}(\omega)} \) round \( E_I \cap E_j \neq \emptyset, j \not\in I \). Assume that we can find a twisted cycle \( C \) in \( H_{n-r}(E_I^\circ, L^\vee) \) such that
\begin{equation} \label{eq-residue-form}
    \int_{C} \textnormal{Res}_{E_I}(\omega) \neq 0.
\end{equation}

As in the previous section, if \( p : F \longrightarrow E_I^\circ \) is the finite cover associated with the local system \( L\), then the twisted cycle \( C \in H_{n-r}(E_I^\circ, L^\vee) \) gets identified with an element in \( H_{n-r}(F, \mathbb{C}) \). The restriction of the morphism \( \rho \) to \( \widetilde{E}_I^\circ \) is a ramified covering of degree \( d_I \) ramified along \( E_I \cap E_j, j \not\in I, \) with monodromies \( \exp{(2 \pi \imath N_j / d_I)} \). On the other hand, the monodromies of \( F \) are \( \exp(2 \pi \imath \epsilon_{I_j}(\omega)) \). From the fact that \( \sigma_I(\omega) \in \mathbb{Z}[\frac{1}{d_I}] \),
\[ \exp(2 \pi \imath \epsilon_{I_j}(\omega)) = \exp{\left(-2 \pi \imath N_j \frac{\alpha_I}{d_I}\right)} = \exp{\left(2 \pi \imath \frac{N_j}{d_I}\right)^{-\alpha_I}}, \]
for some \( \alpha_I \in \mathbb{Z}_{>0} \). Therefore, the restriction of \( \rho \) to \( \widetilde{E}_I^\circ \) factorizes through a connected component \( F_0 \) of \( F \),
\[\restr{\rho}{\widetilde{E}^\circ_I} : \widetilde{E}_I^\circ \xrightarrow{\quad q \quad } F_0 \xrightarrow{\ \ {{p|}_{F_0}}\ \ } E_I^\circ,\]
and there exists \( m \in \mathbb{Z} \setminus \{0\} \) and a cycle \( \tilde{\gamma} \in H_{n-r}(\widetilde{E}^\circ_I, \mathbb{C}) \) such that \( q_* \tilde{\gamma} = m C \). Let \( \delta_i^{(r)} : H_{n-r}(\widetilde{E}^\circ_I, \mathbb{C}) \longrightarrow H_{n}(\widetilde{E}^\circ_i, \mathbb{C}) \) be the \(r\)-fold Leray coboundary operator from \cref{expansions} with target \( \widetilde{E}_i^\circ \). Then, the residue theorem asserts that
\begin{equation} \label{residue-thm-eq}
    \int_{\delta^{(r)}_i \tilde{\gamma}} \restr{R_i(\widetilde{\omega}_0)}{E_i^\circ} = (2 \pi \imath)^{r} \int_{\tilde{\gamma}} \textnormal{Res}_{\widetilde{E}_I}[R_i(\widetilde{\omega}_0)].
\end{equation}
The local triviality of \( \tilde{f} \) in a neighborhood of \( \widetilde{E}_i^\circ \) allows us to extend \( \delta^{(r)}_i \tilde{\gamma} \) to a locally constant cycles \( \tilde{\gamma}(\tilde{t}) \) in \( H_{n}(\widetilde{X}_{\tilde{t}}, \mathbb{C}) \), for \( \tilde{t} \in \widetilde{D}' \), such that \( \tilde{\gamma}(0) = \delta^{(r)}_i \tilde{\gamma} \).

\begin{remark}
    The previous application of the residue theorem under the presence of orbifold singularities can be justified by taking a resolution \( \varpi: \widehat{X} \longrightarrow \widetilde{X} \) of \( \widetilde{X} \) and using the fact that the restriction of \( \varpi^* R_i(\widetilde{\omega}_0) \) has still logarithmic poles along the strict transform of \( \sum_{j \in I \setminus \{i\}} \widetilde{E}_j \cap \widetilde{E}_i \), see \cite[\S 1.17]{Steen76} and \cite[\S 4.4]{Var82a}. Moreover, by \cite[Cor. 1.5]{Steen76}, the map \( \varpi_* : H_{n-r}(\widehat{X}, \mathbb{C}) \longrightarrow H_{n-r}(\widetilde{X}, \mathbb{C}) \) is surjective.
\end{remark}

Setting \( \gamma(\tilde{t}^e) = \rho_* \tilde{\gamma}(\tilde{t}) \), we have
\[ \lim_{t \rightarrow 0} t^{\sigma_{I}(\omega)+1} \int_{\gamma(t)} \frac{\omega}{\dd {f}} = a_{-\sigma_{I}(\omega), 0} \neq 0. \]
Indeed, assuming \eqref{eq-residue-form} is true
\begin{align}
    \nonumber \int_{\gamma(t)} \frac{\overline{\omega}}{\dd \bar{f}} = \int_{\rho_* \tilde{\gamma}(\tilde{t})} \restr{\frac{\overline{\omega}}{\dd \bar{f}}}{\overline{X}_{\tilde{t}^e}} & = \int_{\tilde{\gamma}(\tilde{t})} \rho^* \left(\restr{\frac{\overline{\omega}}{\dd \bar{f}}}{\overline{X}_{\tilde{t}^e}}\right)                                                                      \\
    \nonumber                                                                                                                                                                            & = e^{-1}\tilde{t}^{\,1-e} \int_{\tilde \gamma(\tilde{t})} \restr{\frac{\rho^* \overline{\omega}}{\dd \tilde{f}}}{\widetilde{X}_{\tilde{t}}}                        &   & \eqref{easy-eq}              \\
                                                                                                                                                                                         & = \tilde{t}^{\,e(-\sigma_I(\omega) - 1)} \int_{\tilde{\gamma}(\tilde{t})} e^{-1} R_i(\widetilde{\omega}_0) + \cdots                                                 &
                                                                                                                                                                                         & \eqref{pullback-eq2}                                                                                                                                                                                  \\
    \nonumber                                                                                                                                                                            & = \tilde{t}^{\,e(-\sigma_I(\omega) - 1)} \int_{\delta_i^{(r)}\tilde{\gamma}} e^{-1} R_i(\widetilde{\omega}_0) + \cdots             &   & \textnormal{(loc. constant)} \\
    \nonumber                                                                                                                                                                            & = (2 \pi \imath)^{r}\, \tilde{t}^{\,e(-\sigma_I(\omega) - 1)} \int_{\tilde{\gamma}} e^{-1} \textnormal{Res}_{\widetilde{E}_I}[R_i(\widetilde{\omega_0})] + \cdots &   & \eqref{residue-thm-eq}       \\
    \nonumber                                                                                                                                                                            & = (2 \pi \imath)^{r}\, t^{-\sigma_I(\omega) - 1} \int_{\rho_* \tilde{\gamma}} \textnormal{Res}_{E_I} (\omega) + \cdots                                            &   & \eqref{eq-residue-form}      \\
    \nonumber                                                                                                                                                                            & = (2 \pi \imath)^{r}\, t^{-\sigma_I(\omega) -1} \, m \int_{C} \textnormal{Res}_{E_I}(\omega) + \cdots                                                             &   & (q_* \tilde{\gamma} = m C)
\end{align}
where the dots indicate higher-order terms in \(t\).

\jump

\subsection{Logarithmic terms. Part II} \label{sec-log-terms}
We will show next how for \( \sigma_I (\omega), |I| = r + 1, \) a term \( (\ln t )^r \) might appear in the asymptotic expansions. Instead of using the Leray coboundary to construct an \(n\)-cycle in the Milnor fiber starting from \(\tilde{\gamma}\) in \(\widetilde{E}^\circ_I\), we can construct locally an \(n\)-dimensional differentiable simplex \(\Gamma(\tilde{t}\,)\) using \(\tilde{\gamma}\) such that the integral of \(R_i (\widetilde{\omega}_0) \) depends on \((\ln \tilde{t})^r \). How to globalize this argument to get a relative vanishing cycle in the whole fiber is not clear in general, however, we will do so for dimension \(n = 1\) in the next section.

\jump

Let \(\widetilde{U}_I \subset \widetilde{X}_I \) be an affine open neighborhood and, for ease of notation, assume that \(I = \{0, \dots, r\} \) and \(i = 0\). Let \(y_0, \dots, y_n \) be local coordinates such that a connected component of \(\widetilde{E}_j\) is defined by \( y_j = 0 \), for \(j = 0, \dots, r\), and a connected component of the fiber \(\widetilde{U}_I \cap \widetilde{X}_{\tilde{t}}\) by \(y_0 \cdots y_r = \tilde{t}\). If \( \Delta_n \) is the standard \(n\)-dimensional simplex, which we can think homeomorphic to \( \Delta_{r} \times \Delta_{n-r} \), we can construct the following differentiable simplex on \( \widetilde{X}_{\tilde{t}} \)
\begin{equation*}
    \begin{split}
        \Gamma(\tilde{t}) : \Delta_n \cong &\, \Delta_r \times \Delta_{n-r} \longrightarrow \widetilde{U}_I \cap \widetilde{X}_{\tilde{t}} \cong \mathbb{C}^r \times \mathbb{C}^{n-r}, \\
        &\quad(\underline{\tau}, \underline{\theta})  \longmapsto \Big(\overline{\Gamma}_0(\underline{\tau}), \dots, \overline{\Gamma}_r(\underline{\tau}), \widetilde{\Gamma}(\underline{\theta})\Big)
    \end{split}
\end{equation*}
where \(\overline{\Gamma} : \Delta_r \longrightarrow \mathbb{C}^r \) is a differentiable simplex that sends the \(k\)--th vertex of the standard simplex \(\Delta_r\) to the point \(p_k = (1,  \overset{(k)}{\dots}, 1, \tilde{t}, 1, \dots, 1) \), and \(\widetilde{\Gamma}\) is one differentiable simplex in a simplicial chain representing the \(n\)-cycle \(\tilde{\gamma}\) from the previous section on the connected component \(\{y_i = 0\}_{i \in I}\) of \(\widetilde{E}_I^\circ\).

\jump

More explicitly, we can take \( \overline{\Gamma}_k(\tau_k) \), for \(k = 1, \dots, r\), to be a segment joining \(1\) and \(\tilde{t}\) and \(\overline{\Gamma}_0(\underline{\tau}) = \tilde{t} / \prod_{k \neq 0} \overline{\Gamma}_k(\tau_k) \) so that the image lies on the fiber \(y_0 \cdots y_r = \tilde{t}\). For simplicity, we can also assume that \( \tilde{\gamma} \) consists just of one differentiable simplex \( \widetilde{\Gamma} \), otherwise we can repeat the same construction for each simplex in a simplicial chain representing \(\tilde{\gamma}\).

\jump

Finally,
\begin{align} 
    \nonumber \int_{\Gamma(\tilde{t})} R_i(\widetilde{\omega}_0) & = \int_{\Gamma(\tilde{t})} \frac{\dd y_1}{y_1} \wedge \cdots \wedge \frac{\dd y_r}{y_r} \wedge \textnormal{Res}_{\widetilde{E}_I}[R_i(\widetilde{\omega}_0)]                                                     \\
    & = \int_{\overline{\Gamma}} \frac{\dd y_1}{y_1} \wedge \cdots \wedge \frac{\dd y_r}{y_r} \int_{\widetilde{\Gamma}} \textnormal{Res}_{\widetilde{E}_I}[R_i(\widetilde{\omega}_0)]              \label{eq-log-terms} \\
    \nonumber                                                    & = \int_{1}^{\tilde{t}} \int_{1}^{\tau_1} \cdots \int_{1}^{\tau_{r-1}} \frac{\dd y_r}{y_r} \frac{\dd y_{r-1}}{y_{r-1}} \cdots \frac{\dd y_1}{y_1} \int_{\rho_* \tilde{\gamma}} \textnormal{Res}_{E_I}(\omega) \\
    \nonumber                                                    & = \frac{\big(\ln \tilde{t}\,\big)^r}{r!} \int_{C} \textnormal{Res}_{E_I}(\omega),
\end{align}
where the last equalities follow from \( \Delta_r = \{ \underline{\tau} \in \mathbb{R}^r \ |\ 0 \leq \tau_r \leq \cdots \leq \tau_1 \leq 1  \}\) and the identity
\[
    \int_{1}^{\tau} \frac{(\ln z)^k}{k!} \frac{\dd z}{z} = \frac{(\ln \tau)^{k+1}}{(k+1)!}.
\]

\subsection{Logarithmic terms for plane curves} \label{sec-log-terms-plane}
For this section, we need to set \(n = 1\), but we will keep using the same notations as in previous sections. In this setting we will construct a relative vanishing cycle in \(H_1^{lf}(X_t \setminus \partial X_t, \mathbb{C})\) such that the corresponding period integral has an asymptotic expansion of order \(t^{\alpha - 1} \ln t\), for certain \(\alpha \in \mathbb{Q}_{> 0}\).

\jump

To globalize the constructions from \cref{sec-log-terms} we will make use of the maximal admissible chains constructed in \cref{admissible}. Hence, let us start with \(\omega \in \Omega^2_{\mathbb{C}^2}\) and \(I = \{i, j\} \subset T,\) such that \(E_i \cap E_j \neq \emptyset \) and \(\sigma_{I}(\omega) = \sigma_i(\omega) = \sigma_j(\omega)\). By \cref{prop-admissible-chains}, we can assume that \(\omega \in \Omega^2\) is such that \(\mathcal{M}(\omega)\) has a maximal admissible chain

\begin{center}
    \begin{picture}(500,60)(-30,-10)
        \put(50,20){\circle*{4}}
        \put(50,20){\line(1,0){50}}
        \put(100,20){\circle*{4}}
        \put(111,20){\makebox(0,0){\(\dots\)}}
        \put(140,20){\makebox(0,0){\(\dots\)}}
        \put(150,20){\circle*{4}}
        \put(150,20){\line(1,0){50}}
        \put(200,20){\circle*{4}}
        \put(200,20){\line(1,0){50}}
        \put(250,20){\circle*{4}}
        \put(261,20){\makebox(0,0){\(\dots\)}}

        \put(289,20){\makebox(0,0){\(\dots\)}}
        \put(300,20){\circle*{4}}
        \put(300,20){\line(1,0){50}}
        \put(350,20){\circle*{4}}

        \put(50,32){\makebox(0,0){\(E_{k_0}\)}}
        \put(100,32){\makebox(0,0){\(E_{k_1}\)}}
        \put(202,32){\makebox(0,0){\(E_i = E_{k_i}\)}}
        \put(300,32){\makebox(0,0){\(E_{k_{p-1}}\)}}
        \put(350,32){\makebox(0,0){\(E_{k_p}\)}}
    \end{picture}
\end{center}
containing \(E_i\). Around each intersection point \(\widetilde{E}_{k_j} \cap \widetilde{E}_{k_l}\) between adjacent divisors in the maximal admissible chain we consider small polydisks \(\widetilde{D}_{j, l}\) such that the parts of the Milnor fiber \(\widetilde{X}_{\tilde{t}}\) over them have local equations \(y_j y_l = \tilde{t}\). Since the intersection \(E_{k_j} \cap E_{k_l}\) is just a point, \eqref{eq-residue-form} always holds, and the construction from \cref{sec-log-terms} gives a differentiable \(1\)-simplex \(\Gamma_{j, l}(\tilde{t})\) whose image lies in one of the connected components of \( \widetilde{X}_{\tilde{t}}\) over \(\widetilde{D}_{j, k}\).

\jump

We will start constructing the cycle from the endpoints of the maximal admissible chain, which are either rupture divisors or arrowheads. If an endpoint \(E_{k_0} \) of the admissible chain is a rupture divisor, the maximality of the chain implies that there exists at least one divisor \(E_{k_{0, r}}\) adjacent to \(E_{k_0}\) such that \(\epsilon_{k_{0, r}}(\omega) \in \mathbb{Q} \setminus \mathbb{Z}\). Let \(L\) be the local system introduced in \cref{multivalued} associated with \(E_{k_0}\) and \(\sigma_{k_0}(\omega)\). Denote \(C_{k_0}\) a twisted cycle on \(E_{k_0}\) that loops once counterclockwise around the point \( q_{0, r} = E_{k_0} \cap E_{k_{0, r}}\) and that starts and ends in the boundary of the polydisk \(D_{0, r}\). 

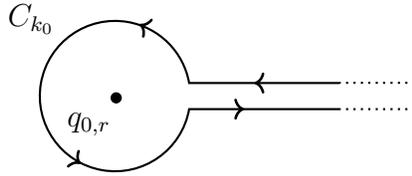
\begin{center}
\begin{tikzpicture}
[decoration={markings,
mark=at position 1.4cm with {\arrow[line width=1pt]{<}},
mark=at position 4cm with {\arrow[line width=1pt]{<}},
mark=at position 7cm with {\arrow[line width=1pt]{<}},
mark=at position -1cm with {\arrow[line width=1pt]{<}},
}
]

\draw[dotted, line width=0.8pt] (4,0) -- (3,0);
\path[draw,line width=0.8pt,postaction={decorate}] +(3,0) -- +(1,0) arc (-10:-350:1) -- +(2, 0);
\draw[dotted, line width=0.8pt] (4,0.35) -- (3,0.35);

\node at (0.03,0.15) [circle,fill,inner sep=1.5pt]{};
\node[below left] at (0.1,0.1) {\(q_{0, r}\)};
\node at (-1.1,1.2) {\(C_{k_0}\)};
\end{tikzpicture}
\jump
\end{center}

Arguing as in \cref{multivalued}, \(C_{k_0}\) defines a cycle \(\tilde{\gamma}_{k_0}\) in a connected component \(F_0\) of \(F\). Transporting the cycle \(\tilde{\gamma}_{k_0}\) in \(F_0 \setminus \widetilde{D}_{0, 1}\) to the nearby fiber, we obtain a family of cycles \(\tilde{\gamma}_{k_0}(\tilde{t})\) with limit \(\tilde{\gamma}_{k_0}\). Due to the non-trivial monodromy around the point \(q_{0, r}\), each endpoint of the cycle in the boundary will lie on different connected component of the fiber \(\widetilde{X}_{\tilde{t}}\) over \(\widetilde{D}_{0, 1}\). By using one differentiable simplex \(\Gamma_{0, 1}(\tilde{t})\) from \cref{sec-log-terms} in each component we can extend this cycle to the part of the fiber over the next divisor \(E_{k_1}\) in the admissible chain.

\jump

We can now extend both ends of this cycle to the parts of the fiber \(\widetilde{X}_{\tilde{t}}\) over the subsequent divisors in the admissible chain. To do so, lift to the fiber straight paths in \(E^{\circ}_{k_j},\) for all \(j = 1, \dots, p\), connecting the boundaries of the polydisks around the intersections and joining them with the differentiable simplices from \cref{sec-log-terms}. Eventually, we must reach an endpoint \(E_{k_p}\) in the maximal admissible chain.

\jump

If an endpoint \(E_{k_p}\) is an arrowhead, the fiber \(\widetilde{X}_{\tilde{t}}\) over the boundary of a small enough polydisk \(\widetilde{D}_{p-1, p}\) is contained in the boundary \(\partial \widetilde{X}_{\tilde{t}}\) of the fiber. Hence, the resulting cycle defines a relative vanishing cycle. Otherwise, \(E_{k_p}\) is rupture exceptional divisor and the maximality of the chain implies, as before, that there is \(E_{k_p, s}\) adjacent to \(E_{k_p}\) such that \(\epsilon_{k_p,s}(\omega) \in \mathbb{Q} \setminus \mathbb{Z}\). In this case, the cycle can be continued by making the limiting cycle do a double loop around \(q_{0, r}\) and \(q_{p, s} = E_{k_p} \cap E_{k_p, s}\),

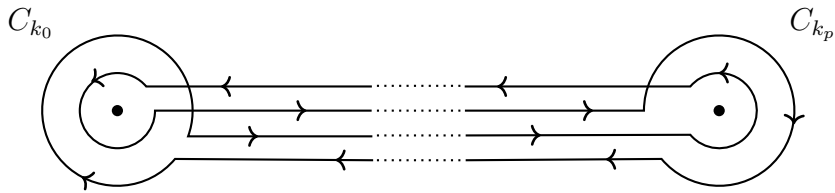
\begin{center}
\begin{tikzpicture}
\draw[dotted, line width=0.8pt] (2.45,0) -- (3.56,0);
\draw[dotted, line width=0.8pt] (2.45,0.33) -- (3.56,0.33);
\draw[dotted, line width=0.8pt] (2.45,-0.33) -- (3.56,-0.33);
\draw[dotted, line width=0.8pt] (2.45,-0.66) -- (3.56,-0.66);

\path[draw,line width=0.8pt,postaction={decorate, decoration={markings,
mark=at position 1cm with {\arrow[line width=1pt]{<}},
mark=at position 5cm with {\arrow[line width=1pt]{<}},
mark=at position 6.8cm with {\arrow[line width=1pt]{<}}
}}] +(2.38, 0) -- ([shift={(-0:0.5cm)}]-1,0) arc (-0:-320:0.5) -- +(3, 0);

\path[draw,line width=0.8pt,postaction={decorate, decoration={markings,
mark=at position 0.5cm with {\arrow[line width=1pt]{>}},
mark=at position 4cm with {\arrow[line width=1pt]{>}},
mark=at position 9.5cm with {\arrow[line width=1pt]{>}}
}}] +(2.4, -0.66) -- ([shift={(-40:1cm)}]-1, 0) arc (-40:-380:1) -- +(2.45, 0);

\path[draw,line width=0.8pt,postaction={decorate, decoration={markings,
mark=at position 1cm with {\arrow[line width=1pt]{>}},
mark=at position 5cm with {\arrow[line width=1pt]{>}},
mark=at position 8cm with {\arrow[line width=1pt]{>}}
}}] +(3.63, -0.33) -- ([shift={(220:0.5cm)}]7,0) arc (220:500:0.5) -- +(-3, 0);

\path[draw,line width=0.8pt,postaction={decorate, decoration={markings,
mark=at position 2cm with {\arrow[line width=1pt]{<}},
mark=at position 5cm with {\arrow[line width=1pt]{<}},
mark=at position 9cm with {\arrow[line width=1pt]{<}}
}}] +(3.63,-0.66) -- ([shift={(220:1cm)}]7, 0) arc (220:540:1) -- +(-2.38,0);

\node at (-1,0) [circle,fill,inner sep=1.5pt]{};
\node at (7,0) [circle,fill,inner sep=1.5pt]{};
\node[below left] at (-1.7,1.5) {\(C_{k_0}\)};
\node[below left] at (8.7,1.5) {\(C_{k_p}\)};
\end{tikzpicture}
\end{center}

In this case, the constructed cycle defines a non-relative vanishing cycle in \(\widetilde{X}_{\tilde{t}}\).

\begin{proposition} \label{prop-log-cycles}
With the notations above, let us denote \(\tilde{\gamma}(\tilde{t})\) the concatenation of all the cycles in \(\widetilde{X}_{\tilde{t}}\) constructed following the admissible chain in \(\mathcal{M}(\omega)\). Then, \(\tilde{\gamma}(\tilde{t})\) defines a class in \(H_1^{lf}(\widetilde{X}_{\tilde{t}} \setminus \partial \widetilde{X}_{\tilde{t}}, \mathbb{C})\) and setting \( \gamma(\tilde{t}^e) = \rho_* \tilde{\gamma}(\tilde{t}) \),
\[
    \int_{\gamma(t)} \frac{\omega}{\dd f} = a_{-\sigma_{I}(\omega), 1}\, t^{-\sigma_{I}(\omega) - 1}\, \frac{\ln t}{2} + a_{-\sigma_{I}(\omega), 0}\, t^{-\sigma_{I}(\omega) - 1} + \cdots
\]
with \(a_{-\sigma_{I}(\omega), 1}\neq 0\), and where the dots indicate higher-order terms in \(t\).
\end{proposition}
\begin{proof}
It follows from the discussion in \cref{sec-log-terms} that the integral of \(\widetilde{\omega}/ \dd \tilde{f}\) along \(\tilde{\gamma}(\tilde{t})\) will have a logarithmic term of the form \(t^{-\sigma_I(\omega)-1}\, \ln t\). What remains to be seen is that the sum of the local contributions at each intersection \(E_{k_j} \cap E_{k_{j+1}} \) in the admissible chain is not zero. 

\jump

Locally, the integral on the part of the fiber \(\widetilde{X}_{\tilde{t}}\) over the small polydisk \(\widetilde{D}_{j, j+1}\) has a non-zero logarithmic term \(t^{-\sigma_I(\omega)-1}\, \ln t\) since \(E_{k_{j}} \cap E_{k_{j+1}}\) is just a point and \eqref{eq-residue-form} holds. To show that the sum of all local terms is not zero, it is enough to focus on one of the paths joining the endpoints of the admissible chain. Indeed, the presence of monodromy around the points \(q_{0, r}\) and \(q_{p,s}\) imply that the integrals of \(\widetilde{\omega} / \dd \tilde{f}\) along two such paths will not cancel out.

\jump

Let \(E_{k_j}, E_{k_{j+1}}\) be a pair of adjacent divisors in the maximal admissible path. After rescaling we can assume the integral over \(\Gamma_{j,j+1}(\tilde{t})\) is given by \eqref{eq-log-terms}. Let \(a^{j}_{-\sigma_{I}(\omega), 1}\) be the coefficient of \(t^{-\sigma_I(\omega) - 1}\, \ln t\) when integrating over \(\Gamma_{j, j+1}(\tilde{t})\). In this case, this coefficient equals 
\[
    \int_{C} \textnormal{Res}_{E_I}(\omega),
\]
which, since \(E_{\{k_j, k_{j+1}\}}\) is just a point and \(\sigma_{k_j}(\omega) = \sigma_{k_{j+1}}(\omega)\), it is just the coefficient \(a^{j+1}_{(0,0)}\) of the lowest order term in \(y_0 y_1\) of \(\widetilde{\omega} = \tilde{g}(y_0, y_1)\, \dd y_0 \dd y_1\), for some local coordinates \(y_0, y_1\) around \(\widetilde{E}_{k_{j}} \cap \widetilde{E}_{k_{j+1}}\) that satisfy \(\tilde{f} = y_0 y_1 = \tilde{t}\). 

\jump

We will show next that, for all intersections in the admissible chain, all local contributions are equal. To show this, it is enough to consider \(E_{k_{j-1}}\) another divisor in the admissible chain intersecting \(E_{k_j}\). Assume that, locally around \(\widetilde{E}_{k_{j}} \cap \widetilde{E}_{k_{j+1}}\), the divisor \(\widetilde{E}_{k_j}\) is defined by \(y_0 = 0\), and the intersection point \(\widetilde{E}_{k_j} \cap \widetilde{E}_{k_{j-1}}\) is at \(y_1 = \lambda \in \mathbb{C}^*\). In these local coordinates, and for the sake of this argument, it is enough to assume that
\[
    \tilde{f} = y_0 y_1 (c + b y_1), \qquad \widetilde{\omega} = y_0^{v_j} y_{1}^{v_{j+1}} (c + b y_1)^{v_{j-1}} \dd y_0 \dd y_1,
\]
locally around \(\widetilde{E}_{k_j} \cap \widetilde{E}_{k_{j+1}}\), with \(\lambda = -b/c\), and where we denoted \(v_s = e(k_s + v_s(\omega))/N_s - 1\) to simplify the notations. Let us first compute \(a^{j+1}_{-\sigma_{I}(\omega), 1}\) explicitly. Consider the change of coordinates \(\widetilde{y}_1 = y_1 (c + b y_1)\). Then,
\[
    \tilde{f} = y_0 \widetilde{y}_1, \qquad \widetilde{\omega} = y_0^{v_j} \widetilde{y}_1^{v_{j+1}} \widetilde{u}(y_0, \widetilde{y}_1)^{v_{j+1} + 1} c^{v_{j-1}} \dd y_0 \dd \widetilde{y}_1 + \cdots,
\]
where the dots indicate higher-order terms in \(y_0\) or \(\tilde{y}_1\), and \(\widetilde{y}_1 \widetilde{u}(y_0, \widetilde{y}_1)\) is the functional inverse of \(y_1 (c + b y_1)\). In particular, \(\widetilde{u}(0,0) = c^{-1}\), and the integration path gets transformed to \(c\, \Gamma_{j, j+1}(\tilde{t})\). Therefore, \(a^{j+1}_{(0,0)} = c^{-1}\), since we are assuming \(\sigma_{k_{j+1}}(\omega) = \sigma_{k_{j-1}}(\omega)\), and hence \(a^{j+1}_{-\sigma_I(\omega), 1} = 1\). Now, performing the change of coordinates \( z_1 = c + b y_1\),
\[
    \tilde{f} = b^{-1} y_0 z_1 (z_1 - c), \qquad \widetilde{\omega} = b^{-v_{j+1}-1} y_0^{v_j} z_1^{v_{j-1}}  (z_1 - c)^{v_{j+1}} \dd y_0 \dd z_1,
\]
and the integration path is \(b\, \Gamma_{j,j-1}(\tilde{t})\). The same computation shows that \(a_{(0,0)}^{j-1} = -c^{-1}b^{-1}\) and, integrating along \(-c b \Gamma_{j, j-1}(\tilde{t})\) one gets that \(a^{j-1}_{-\sigma_I(\omega),1} = 1\), as required. The case where \(\lambda = \infty\) can be proven similarly by performing the appropriate change of coordinates, namely \(w_0 = z_1 x^{\kappa}_0, z_1 y_0 = 1\), where \(\kappa = -E_{k_j} \cdot E_{k_j}\). 
\end{proof}

\begin{proof}[Proof of \cref{thm-main-2}]
    We will begin by showing that the results are independent of the differential form \(\omega = g \dd x\) chosen in \cref{thm-main-1}. 
    
    \jump
    
    Since \(E_i\) is rupture, \cref{cor2-prop-rup2} implies that \(\epsilon_{i_1}(\omega) = 0\), respectively \(\epsilon_{i_2}(\omega) = 0\), if and only if there is at least one arrowhead in the connected component of \(\Delta_{\pi} \setminus \{E_i\}\) containing \(E_{i_1}\), respectively \(E_{i_2}\). Similarly, by \cref{cor-eps-proposition,cor-eps-prop}, \(\epsilon_{i_j}(\omega) = 0 \), for some \(j \in \{3, \dots, m_i\}\), if and only if we are not in the case \((\dag)\). Hence, the conclusion in the first part only depends on the minimal log resolution. For the second part, it is clear that, as long as \(\sigma_i(\omega)\) remains constant, the condition \(\epsilon_{i_j}(\omega) \in \mathbb{Q} \setminus \mathbb{Z}\) is independent of the differential form chosen. 
    
    \jump
    
    Let \(\omega = g \dd x \in \Omega^2_{\mathbb{C}^2, x}\) such that \(v_i(g) < N_i\). If \(\omega\) has one residue number equal to zero, then, by \cref{prop-admissible-chains}, we can assume that \(\mathcal{M}(\omega)\) contains a maximal admissible chain. Therefore, \cref{prop-log-cycles} implies that there exists a relative cycle \(\gamma(t) \in H^{lf}_1(X_t \setminus \partial X_t, \mathbb{C})\) such that the integral of \(\omega/\dd f\) along \(\gamma(t)\) depends on \(t^{-\sigma_i(\omega) - 1}\,\ln t\). Consequently, \(\sigma_i(\omega)\) is a root of \(b_{f,x}(s)\) with multiplicity 2 by \cref{prop-bfct}.
    
    \jump

    Assume now that \(\omega\) has at least three residue numbers in \( \mathbb{Q} \setminus \mathbb{Z}\) and none of the others vanish, as otherwise we would be in the previous case. \cref{thm-main-1} implies that all residue numbers in \(\mathbb{Z}\) must be equal to one. In addition, there is at least one \(\epsilon_{i_j}(\omega) \in \mathbb{Q}\setminus\mathbb{Z}\), for some \(j \in \{3, \dots, m_i\}\). If the form \(R_i(\omega)\) has an extra zero on \(E_i^\circ\), that is, with the notation from \cref{sec-bounds-residue} \(\alpha_{m_i+1} > 0\), we can construct another differential form \(\omega' \in \Omega^2_{\mathbb{C}^2, x}\) such that \(\epsilon_{i_j}(\omega') = \epsilon_{i_j}(\omega) + \alpha_{m_i+1} \in \mathbb{Q} \setminus \mathbb{Z}\), for some \(j \in \{3 \dots, m_i\}\), and still \(\sigma_i(\omega') = \sigma_i(\omega)\). Indeed, we can replace the factor \(f_i^{\gamma_i}\), with \(\gamma_i = \alpha_{m_i+1}\), in \(\omega\) by the factor \(f_{v,j}^{\gamma_i}\), see \cref{sec-bounds-residue}. 
    
    \jump
    
    Then, ignoring those intersections with residue numbers equal to one and after \cref{lemma-sumeps}, the hypotheses of \cref{deligne-mostow} hold. By \cref{multivalued}, there exists a vanishing cycle \(\gamma(t) \in H_1(X_t, \mathbb{C})\) such that the integral of \(\omega' / \dd f\) along \(\gamma(t)\) depends on \(t^{-\sigma_{i}(\omega) - 1}\). Consequently, \cref{prop-bfct} implies that \(\sigma_i(\omega)\) is a root of \(b_{f,x}(s)\) with multiplicity 1.
    
    \jump

    For a non-exceptional divisors \(E_i\) with multiplicity \(N_i\) we can always find a point \(p\) in \(X_0\) and local coordinates such that, locally, \(f = x_0^{N_i}\). In this case, it is well-known that the Bernstein-Sato polynomial of \(f\) at \(p\) has roots \(-\frac{1}{N_i}, \dots, -\frac{N_i-1}{N_i}, -1\), which are roots of geometric origin since for \(\omega = f^\nu \dd x\), one has \(\sigma_i(\omega) = -\frac{\nu + 1}{N_i}\). Finally, since \(b_{f, p}(s)\) must divide \(b_{f, x}(s)\), the last claim follows.
\end{proof}

\begin{remark} \label{rmk-multiplicity}
    Since for any exceptional divisor \(E_i\), \(v_i(f^\nu \dd x) = \nu N_i\), the only possible topological root coming from a non-exceptional divisor that can be a double root in \cref{thm-main-2} is \(-\frac{1}{N_i}\). In this case, the multiplicity will be detected by the exceptional divisor adjacent to the strict transform, except for the trivial case \(f = x^Ny^M \). Moreover, if \(-\frac{1}{N_i}\) has multiplicity 2, then it is necessarily the log canonical threshold at \(x \in \mathbb{C}^2\), see \cite[Thm. 4.2]{Vey95}.
\end{remark}

Since irreducible plane curves have semisimple monodromy, see \cite{Tra72}, the reduced Bernstein-Sato polynomial has no multiple roots and, for this case, we have the following closed expression for the topological roots given by \cref{thm-main-2}.

\begin{corollary} \label{cor-thm-2}
    Let \(f : (\mathbb{C}^2, x) \longrightarrow (\mathbb{C}, 0)\) be an irreducible plane curve with semigroup \(\Gamma = \langle \overline{\beta}_0, \overline{\beta}_1, \dots, \overline{\beta}_g \rangle\). Then, the set of topological roots of geometric origin from \cref{thm-main-2} is
    \[
        \bigcup_{i= 1}^g \left\{ \sigma_\nu = - \frac{m_i + n_1 \cdots n_i + \nu}{n_i \overline{\beta}_i} \ \bigg|\ \nu \in \Gamma_i,\, 0 \leq \nu < n_i \overline{\beta}_i,\, e_{i-1} \sigma_{\nu}, \overline{\beta}_i \sigma_\nu \not\in \mathbb{Z} \right\} \cup \{-1\},
    \]
    where \(\Gamma_i \) is the semigroup \(\langle n_1 \cdots n_{i}, n_2 \cdots n_{i} \overline{m}_1, \dots, n_i \overline{m}_{i-1}, \overline{m}_i \rangle\).
\end{corollary}

\section{Relation with known roots}

\subsection{Jumping numbers and multiplier ideals} \label{sec-multipliers}

Let \(X\) be a complex algebraic variety. To any hypersurface singularity defined by \(f \in \mathscr{O}_{X}(X)\), or more generally any \(\mathbb{Q}\)-divisor, one can associate a family of ideal sheaves \(\mathcal{J}(f^\lambda)\), the multiplier ideals, parametrized by positive rational numbers \(\lambda \in \mathbb{Q}_{>0}\). These ideals form a nested sequence
\[
 \mathscr{O}_{X} \supsetneq \mathcal{J}(f^{\lambda_1}) \supsetneq \mathcal{J}(f^{\lambda_2}) \supsetneq \cdots \supsetneq \mathcal{J}(f^{\lambda_i}) \supsetneq \cdots
\]
and the rational numbers \(0 < \lambda_1 < \lambda_2 < \dots < \lambda_i < \dots \) where there is a strict inclusion are called the jumping numbers.

\jump

Let \(\pi: Y \longrightarrow X \) be a log resolution of the hypersurface defined by \(f\). Then, the multiplier ideals have the following birational definition
\[
    \mathcal{J}(f^\lambda) = \pi_* \mathscr{O}_Y (\lceil K_{\pi} - \lambda F_\pi \rceil),
\]
see \cite{Laz04-2}. For other characterizations of multiplier ideals see \cite{ELSV04,BS05,MP19}.

\jump

We are mainly interested in the case where \(X\) is a smooth complex surface and \(f \in \mathscr{O}_{X, x}\) is a germ of a holomorphic function at some point \(x \in X\). In this situation, we can make use of the theory of jumping divisors from \cite[\S 4]{AAD16}. A reduced divisor \(G \leqslant F_\pi\) is a jumping divisor for a jumping number \(\lambda \in \mathbb{Q}_{>0}\) if
\[
    \mathcal{J}(f^{\lambda - \epsilon}) = \pi_* \mathscr{O}_{Y}(\lceil K_\pi - \lambda F\rceil + G),
\]
for \(\epsilon \ll 1\). A jumping divisor is minimal if no proper subdivisor is a jumping divisor for \(\lambda\), i.e.
\[
    \mathcal{J}(f^{\lambda - \epsilon}) \supsetneq \pi_* \mathscr{O}_{Y}(\lceil K_\pi - \lambda F\rceil + G),
\]
for any \(0 \leqslant  G' < G\). The existence and uniqueness of a minimal jumping divisor \(G_\lambda\) for any jumping number \(\lambda \in \mathbb{Q}_{>0}\) follows from \cite[Prop. 4.4, Cor. 4.6]{AAD16}. The minimal jumping divisor is defined as the reduced divisor supported on the components \(E_i\) of \(F_\pi\) such that 
\[
    \lambda = \frac{k_i + v_i}{N_i},
\]
where the multiplicities \(v_i\) come from a certain antinef divisor \(D = \sum_{i \in T} v_i E_i \). 

\begin{proposition}[{\cite[Prop. 4.16, Lemma 4.14]{AAD16}}]
For any component \(E_i \leqslant G_\lambda\) of the minimal jumping divisor \(G_\lambda\) associated with a jumping number \(\lambda \in \mathbb{Q}_{>0}\) one has
\[
    (\lceil K_\pi - \lambda F_\pi \rceil + G_\lambda) \cdot E_i \geq 0.
\]
Moreover, 
\[
    (\lceil K_{\pi} - \lambda F_{\pi}\rceil + G_\lambda) \cdot E_i = -2 + \sum_{j = 1}^{r} \{\lambda N_{i_j} - (k_{i_j} - 1)\} + r_{G_\lambda}(E_i),
\]
where \(\{\cdot\}\) denotes the fractional part, and \(r_{G_\lambda}(E_i)\) is the number of divisors different from \(E_i\) in the support of \(G_\lambda\) and adjacent to \(E_i\).
\end{proposition}

Using these results we can prove that jumping numbers in the interval \((0, 1)\) are described by \cref{thm-main-2}.

\begin{proposition} \label{prop-jumping}
    The opposites of the jumping numbers in \( (0, 1)\) are contained in the set of topological roots of the local Bernstein-Sato polynomial described by \cref{thm-main-2}.
\end{proposition}
\begin{proof}
Since the divisor \(D\) defining the minimal jumping divisor \(G_\lambda\) is antinef, we can find \(g \in \mathscr{O}_{X, x}\) such that \(D = \textnormal{Div}(\pi^* g)\). Letting \(\omega = g \dd x\) we have that \( \lambda = -\sigma_i(\omega) \) for any \(E_i\) in the support of \(G_\lambda\). Since \(\lambda \in (0, 1)\), we have that \(v_i(g) < N_i\).

\jump

From the previous proposition, it follows that
\[
    (\lceil K_{\pi} - \lambda F_\pi \rceil + E_i) \cdot E_i = -2 - \sum_{j = 1}^r \{ \epsilon_{i_j}(\omega)\} + r_{G_\lambda}(E_i) \geq 0.
\]
If \(E_i\) is an isolated component in the support of \(G_\lambda\), then \(r_{G_\lambda}(E_i) = 0\) and there must exist at least three \(\epsilon_{i_j}(\omega) \in \mathbb{Q} \setminus \mathbb{Z} \). This implies that \(E_i\) is rupture and the second case of \cref{thm-main-2} is satisfied. Otherwise, there are at least two components \(E_i\), \(E_j\) in the support of \(G_\lambda\) intersecting each other and such that \(\sigma_i(\omega) = \sigma_j(\omega)\). Therefore, the first case of \cref{thm-main-2} follows. 
\end{proof}

\subsection{Poles of the motivic zeta function} \label{sec-motivic}

Let \(X\) be a complex algebraic variety and \(f: X \longrightarrow \mathbb{C}\) a non-constant regular function. For a non-negative integer \(n\) and a point \(x \in X\) consider
\[
    \mathcal{L}_n(X) = \textnormal{Hom}(\textnormal{Spec}\, \mathbb{C}\llbracket t \rrbracket / (t^n) , X), \qquad \mathcal{L}_{n, x}(X) = \{\gamma(t) \in \mathcal{L}_n(X) \ | \ \gamma(0) = x \},
\]
the arc space of order \(n\) on \(X\) and the arc space of order \(n\) on \(X\) centered at \(x\), respectively. The inverse limits \(\mathcal{L}(X)\) and \(\mathcal{L}_{x}(X)\) define the arc space on \(X\) and the arc space on \(X\) centered at \(x\).  Then, the \(n\)-th contact loci of \(f\) is defined as
\[
    \mathcal{X}_n = \{ \gamma(t) \in \mathcal{L}_n(X) \ | \ \textnormal{ord}_t (f\circ \gamma) = n \},
\]
and similarly for its local counterpart \(\mathcal{X}_{n, x}\) at \(x \in X\). If \(K_0(\textnormal{Var}_\mathbb{C})\) denotes the Grothendieck group of varieties over \(\mathbb{C}\) and \([X]\) denotes the class of a variety \(X\), the local motivic zeta function of \(f\) at \(x \in X\) is defined as the motivic integral 
\[
    Z_{\textnormal{mot}, x}(f; s) = \int_{\mathcal{L}_x(X)} \mathbb{L}^{-s\, \textnormal{ord}_t f} d \mu = \sum_{n \in \mathbb{Z}} [\mathcal{X}_{n, x}] \mathbb{L}^{-n s}, \qquad \textnormal{Re}(s) > 0,
\]
in the ring \(K_0(\textnormal{Var}_{\mathbb{C}})\llbracket\mathbb{L}^{-1}\rrbracket\). Motivic integration and the motivic zeta function are generalizations of \(p\)-adic integration and the \(p\)-adic Igusa zeta function since the former specialize, for almost all primes \(p\), to the latter. The reader is referred to the surveys \cite{Veys06,Loe09} for an overview of these theories and their applications to problems in number theory and birational geometry.

\begin{theorem}[Denef-Loeser \cite{DL98}] \label{thm-denef-loeser}
Let \(\pi : Y \longrightarrow X\) be a log resolution of \(\{f = 0\}\), then
\[
    Z_{\textnormal{mot}, x}(f; s) = \frac{1}{\mathbb{L}^n} \sum_{J \subseteq T} [E_J^\circ \cap \pi^{-1}(x)] \prod \frac{\mathbb{L} - 1}{\mathbb{L}^{N_i s + k_i} - 1}.
\]
In particular, this implies that \(Z_{\textnormal{mot}, x}(f;s) \) is a rational function in \(\mathbb{L}^{-1}\).
\end{theorem}
A similar expression holds for the global version \(Z_{\textnormal{mot}}(f; s)\). From \cref{thm-denef-loeser} we see that the candidate poles for the local motivic zeta functions are \(\sigma_i = -k_i/N_i, i \in T\). Since the ring \(K_0(\textnormal{Var}_{\mathbb{C}})\) is not a domain some care must be taken to define the order of a pole, see \cite[Remark 3.7]{NX16} for the precise definition.

\begin{conjecture}[Strong Monodromy Conjecture]
Let \(f : X \longrightarrow \mathbb{C}\) a non-constant regular function and \(x \in X\), then
\begin{enumerate}[label=\roman*)]
\item If \(s_0\) is a pole of \(Z_{\textnormal{mot}, x}(f; s)\), then \(s_0\) is a root of \(b_{f, x}(s)\).
\item The order of a pole of \(Z_{\textnormal{mot}, x}(f;s)\) is at most its multiplicity as a root of \(b_{f, x}(s)\).
\end{enumerate}
\end{conjecture}
One can state a similar conjecture for the global motivic zeta function \(Z_{\textnormal{mot}}(f;s)\) and the Bernstein-Sato polynomial \(b_f(s)\), and the local conjecture implies the global. There is also a weak version of this conjecture predicting a connection between the poles of \(Z_{\textnormal{mot}, x}(f; s)\) and the eigenvalues of the local monodromy on the cohomology of the Milnor fibers at some point of \(\{f = 0\}\) near \(x \in X\). By results of Malgrange \cite{Mal83} and Kashiwara \cite{Kas83}, the strong version implies the weak. Finally, one can announce analogous conjectures for the \(p\)-adic Igusa zeta function or the topological zeta function.

\jump

Focusing now in the case where \(f : (\mathbb{C}^2, x) \longrightarrow (\mathbb{C}, 0)\), the actual poles of \(p\)-adic Igusa zeta function were determined in \cite{Veys90}, however, the proof generalizes directly to the motivic setting.

\begin{theorem}[{\cite[Thm. 3.2]{Veys90}}] \label{thm-veys}
Let \(s_0 \in \{-k_i/N_i\ |\ i\in T_e\}\). Then, \(s_0\) is a pole of \(Z_{\textnormal{mot}, x}(f;s)\) if and only if there is a rupture divisor \(E_i\) such that \(s_0 = -k_i/N_i\). Moreover,
\begin{enumerate}[label=\roman*)]
\item \(s_0\) is a pole of order 2 if and only if at least one intersection \(E_{i} \cap E_{j} \neq \emptyset \) occurs such that \(k_i/N_i = k_j/N_j\).
\item \(s_0\) is a pole of order 1 if and only if no such intersection occurs and there is at least one \(E_i\) rupture such that \(s_0 = -k_i/N_i\).
\end{enumerate}
\end{theorem}

\begin{proof}[Proof of \cref{main-prop}]
The claim about the opposites of the jumping numbers follows directly from \cref{prop-jumping}. The claim about the poles of \(Z_{\textnormal{mot}, x}(f;s)\) is a consequence of \cref{thm-veys} and \cref{thm-main-2} applied to the standard volume form \(\omega = \dd x\).
\end{proof}

In particular, \cref{main-prop} proves both parts of the Strong Monodromy Conjecture in dimension \(2\). The proof of the first part is the same as Loeser's proof in \cite{Loe88}. However, as discussed in the introduction the proof in \cite{Loe88} does not completely cover the second part of the Conjecture when \(f\) is not reduced.

\section{Some families of examples} \label{examples}

We describe next some families of examples of topological classes where the roots described by \cref{thm-main-2} are all the topological roots of \( b_{f, x}(s) \) for all curves \( (f^{-1}(0), x) \) in the topological class.

\begin{example} \label{ex100}
    Consider the topological class of the plane branch \( f = (y^2-x^3)^2-x^{s+2}y\), with \( s \geq 3 \). The semigroup of this branch is \(\Gamma = \langle 4, 6, 2s + 7 \rangle\). The divisors associated with the minimal log resolution are
    \[
        F_{\pi, exc} = 4E_{0} + 6E_1 + 12E_2 + \cdots + (4s + 14)E_{s+1}, \quad K_\pi = E_0 + 2E_1 + 4E_2 + \cdots + (2s + 5)E_{s+1}.
    \]
    The rupture divisors are \( E_2 \) and \( E_{s+1} \), and the associated valuations have semigroups \( \Gamma_2 = \langle 2, 3 \rangle  \) and \( \Gamma_{s+1} = \Gamma \). Then, the topological roots of geometric origin given by \cref{thm-main-2} are
    \begin{displaymath}
        \begin{gathered}
            E_2 : -\frac{5}{12}, -\frac{7}{12}, -\frac{11}{12}, -\frac{13}{12},\\
            E_{s+1} :  \left\{ \sigma_{s+1, \nu} = - \frac{2s + 5 + \nu}{4s + 14} \ \bigg| \ \nu \in \Gamma, 0 \leq \nu < 4s + 14,\, (2s + 6)\sigma_{s+1, \nu}, (2s + 7)\sigma_{s+1, \nu} \not\in \mathbb{Z} \right\},
        \end{gathered}
    \end{displaymath}
    and \( -1 \). One can check that these are, in fact, all roots of \( b_{f, x}(s) \). Since the singularity is isolated, the number or roots of \(b_{f,x}(s)\) is bounded by \(\mu + 1\), where \(\mu\) is the Milnor number of \(f\). For this family of examples, \( \mu = 2s + 10 \), and the number of topological roots coming from \(E_{s+1}\) is \(2s + 6\).

    \jump

    The reason for this phenomenon is that, up to analytical equivalence, there is only one curve in each of these topological classes; see \cite[\S II.1]{TeiApp86}. In contrast, the poles of the motivic zeta function are \( -\frac{5}{12}, -\frac{2s + 5}{4s + 14} \), and the jumping numbers in \( (0, 1) \) are \( \frac{5}{12}, \frac{11}{12}, \frac{2s + 9}{4s + 14}, \dots \), for a total of \(\mu/2 = s + 5\).
\end{example}

On the other hand, there are examples of topological classes of plane branches such that any root of \(b_{f,x}(s)\) not described by \cref{thm-main-2} changes within the topological class.

\begin{example} \label{ex101}
    Let \(\Gamma = \langle n, m \rangle\), \(\gcd(n, m) = 1\), be the semigroup of a plane branch with one characteristic exponent. Consider the \(\mu\)-constant family
    \begin{equation} \label{eq-ex101}
        f_{t} = y^n - x^m + \sum_{\substack{ni + mj > nm, \\0 \leq i < m -1, \\ 0 \leq j < n -1}} t_{i, j} x^i y^j.
    \end{equation}
    The topological roots of geometric origin given by \cref{thm-main-2} are
    \[
        S_{top} = \left\{ \sigma_{i, j} = -\frac{ni + mj}{nm} \ \bigg| \ ni + mj < nm + n + m,\quad n\sigma_{i,j}, m\sigma_{i,j} \not\in \mathbb{Z} \right\},
    \]
    and \(-1\). For the quasi-homogeneous curve \(f_0 = y^n - x^m\) it is well-known that
    \[
        \widetilde{b}^{-1}_{f_0, x}(0) = \left\{ -\frac{ni + mj}{nm}\ \bigg|\ 1 \leq i \leq m - 1, 1 \leq j \leq n - 1 \right\}.
    \]
    The roots of \(\widetilde{b}_{f_0, x}(s) \) that are not in \(S_{top}\) are in bijection with the terms \(t_{i,j} x^i y^j \) in \eqref{eq-ex101}. It can be shown, see \cite{CN87}, that if \(\sigma_{i,j} \not\in S_{top}\) is a root of \(\widetilde{b}_{f_0,x}(s)\) then, \(\sigma_{i, j} - 1\) is a root of \(f_t\) if the corresponding parameter \(t_{i, j} \) is different from zero.

    \jump

    In this example, the \(\mu\)-constant family \eqref{eq-ex101} induces the miniversal deformation of the quasi-homogeneous curve \(f_0 = y^n - x^m\). Moreover, this \(\mu\)-constant family covers the whole topological class since, for this particular case, any plane branch in the topological class is analytically equivalent to some fiber of this deformation, see \cite[\S VI]{Zar86}.
\end{example}

It is then reasonable to ask the following question.

\begin{intro-question} \label[question]{the-question}
    For any topological class \(\mathcal{T}\) of plane curves, are the topological roots of geometric origin from \cref{thm-main-2} all the possible topological roots of \(\mathcal{T}\)? More precisely, if for any \((f^{-1}(0), x)\) in \(\mathcal{T}\) the local Bernstein-Sato polynomial has a root \(\sigma\) that is not a topological root of geometric origin, does there exist \((g^{-1}(0), x)\) in \(\mathcal{T}\) such that either \(\sigma - 1\) or \(\sigma + 1\) is a root of \(b_{g,x}(s)\)?
\end{intro-question}

In addition to the examples above, \cref{the-question} is true for more general classes of plane branches. Assuming that the eigenvalues of the monodromy are pair-wise different, for \(g = 2\) a positive answer is given in \cite{ACLM18}. 

\end{document}